\documentclass[bibay2,jsl]{asl}

\usepackage[hidelinks]{hyperref}
\usepackage{bussproofs,bussproofs-extra}

\def\M{\mathfrak{v}}

\def\C{\mathcal{C}}
\let\lif\rightarrow
\let\nif\nrightarrow
\let\xor\oplus
\let\nor\downarrow
\def\ite{{\lif}{/}}
\def\L#1{\mathbf{L #1}}
\def\Lc#1{\mathbf{L^c #1}}
\def\Ls#1{\mathbf{LJ #1}}

\def\Nm#1{\mathbf{N}_m \mathbf{#1}}
\def\Nms#1{\mathbf{N}_m^s\mathbf{#1}}
\def\Nmp#1{\mathbf{N}_m^p\mathbf{#1}}
\def\Nmsl#1{\mathbf{N}_m^l\mathbf{#1}}
\def\Nsl#1{\mathbf{N}^l\mathbf{#1}}
\def\Ns#1{\mathbf{N}^l\mathbf{#1}}
\def\LK{\mathbf{LK}}
\def\LJ{\mathbf{LJ}}

\def\NJ{\mathbf{NJ}}
\def\FD{\mathbf{FD}}
\def\Q{\mathsf{Q}}

\def\cut{\mbox{\textsc{cut}}}
\def\kut{\mbox{\textsc{kut}}}
\def\mix{\mbox{\textsc{mix}}}
\def\kix{\mbox{\textsc{kix}}}
\def\gem{\mbox{\textsc{gem}}}
\def\LEM{\mbox{\textsc{lem}}}

\def\subst{\mbox{\textsc{subst}}}

\let\opa\circledast
\let\opb\odot

\def\Left#1{\mathord{#1}\textsc{l}}
\def\Right#1{\mathord{#1}\textsc{r}}
\def\Intro#1{\mathord{#1}\textsc{i}}
\def\Elim#1{\mathord{#1}\textsc{e}}
\def\LElim#1{\mathord{#1}\textsc{le}}
\def\RElim#1{\mathord{#1}\textsc{re}}
\def\Weak{\textsc{w}}
\def\Exch{\textsc{x}}
\def\Contr{\textsc{c}}

\def\fCenter{\ensuremath{\vdash}}

\def\defaultHypSeparation{\hskip 1em}

\newtheorem{thm}{Theorem}
\newtheorem{lem}[thm]{Lemma}
\newtheorem{prop}[thm]{Proposition}

\newtheorem{conj}[thm]{Conjecture}
\newtheorem{cor}[thm]{Corollary}
\theoremstyle{definition}
\newtheorem{defn}[thm]{Definition}

\title[Cut elimination and normalization for generalized calculi]{Cut elimination and normalization for generalized single and
multi-conclusion sequent
and natural deduction calculi}

\author{Richard Zach}
\address{University of Calgary\\
Department of Philosophy\\
2500 University Drive NW\\
Calgary, AB T2N 1N4, Canada}
\email{rzach@ucalgary.ca}
\urladdr{https://richardzach.org/}

\begin{document}

% This paper appears in the RSL

\def\journalname{The Review of Symbolic Logic}
\def\shortjournalname{The Review of Symbolic Logic}
\def\longjournalname{The \hfil Review \hfil of \hfil Symbolic \hfil
Logic}
\def\jslname{{\bfseries\itshape\selectfont The Journal of Symbolic Logic}}
\def\ISSN{1755-0203}

\begin{abstract}
Any set of truth-functional connectives has sequent calculus rules
that can be generated systematically from the truth tables of the
connectives. Such a sequent calculus gives rise to a
multi-conclusion natural deduction system and to a version of
Parigot's free deduction. The elimination rules are ``general,'' but
can be systematically simplified. Cut-elimination and normalization
hold. Restriction to a single formula in the succedent yields
intuitionistic versions of these systems. The rules also yield
generalized lambda calculi providing proof terms for natural deduction
proofs as in the Curry-Howard isomorphism. Addition of an indirect
proof rule yields classical single-conclusion versions of these
systems. Gentzen's standard systems arise as special
cases.\thanks{Forthcoming in \emph{The Review of Symbolic
Logic},
\href{http://dx.doi.org/10.1017/S1755020320000015}{DOI:10.1017/S1755020320000015}.
\copyright~Cambridge University Press.}
\end{abstract}

\maketitle

\section{Introduction} 

The literature on structural proof theory and proof-theoretic
semantics, and especially the literature on proof-theoretic harmony, is
full of sophisticated considerations offering various methods for
deriving elimination rules from introduction rules, articulating
criteria for what the meaning expressed or inherent in an introduction or
elimination rule is, and of investigations of particular sets of
rules regarding their proof-theoretic properties and strengths.
Rather than attempt to cut through this multitude of complex and
interacting considerations, it may be beneficial to adopt a general
perspective.  Such a general perspective may provide insight into the
combinatorial \emph{reasons} why certain sets of rules and certain
ways of constructing calculi have certain properties (such as cut
elimination or normalization).

Here we explore an approach based on the work of \cite{Baaz1984}. This
approach gives a general method for determining rules for a classical
sequent calculus for arbitrary $n$-valued connectives. \cite{Zach1993}
and \cite{BaazFermullerZach1994} showed that the cut-elimination
theorem holds for such systems. The approach generalizes the classical
case: Gentzen's $\LK$ is the result of the method applied to the
truth-tables for the classical $2$-valued connectives. Thus, the
cut-elimination theorem for $\LK$ can be \emph{explained} by the
structural features of the sequent calculus together with a semantic
feature of the logical rules of~$\LK$, namely that the conditions
under which the premises of any pair of left and right rules are
satisfied are incompatible.  This is an explanation in the sense that
the specific result for~$\LK$ follows from a general result about any
sequent calculus the rules of which satisfy this simple semantic
property. A restricted calculus such as one for intuitionistic logic
will also satisfy this property, and so the explanation extends
to~$\LJ$. It shows that there is nothing \emph{special} about the
usual rules for $\land$, $\lor$, $\lif$; any truth functional operator
can be accommodated using a general schema of which the usual
operators are special cases, and cut-elimination will hold.
\cite{BaazFermullerZach1993b} showed how the same method can be used
to obtain multi-conclusion natural deduction systems. This applies in
the $2$-valued case as well, and yields natural deduction systems for any
truth-functional connective (e.g., see \cite{Zach2016} for a detailed
treatment of the cases of the nand and nor operators). Here we show
that the connections between multi-conclusion and
single-conclusion natural deduction systems generalize, and that all
such systems (multi- or single-conclusion) normalize.

The general result about cut elimination is obtained by considering
the relation between the cut rule in $\LK$ and the resolution rule,
and observing that a resolution refutation yields a derivation of the
conclusion of a cut (more precisely, mix) inference from the premises
of corresponding left and right logical inferences directly, using
only cuts.  This is the essential part of Gentzen's cut-elimination
procedure. It also underlies the reduction and simplification
conversions of the normalization proof for natural deduction.  Rather
than attempt to explain the proof-theoretic harmony enjoyed by
specific sets of rules by deriving some of them from others (e.g., the
derivation of elimination rules from introduction rules), we can
explain it as a result of combinatorial features of the rules which
flow naturally from the classical semantics.  These features are
preserved, moreover, when the calculi are restricted in certain ways
to obtain intuitionistic systems.

Combinatorial features of rules can also explain, and be explained by,
the properties we want the full systems they are part of to enjoy. The
reduction of cuts and local derivation maxima requires, and is made
possible by, a combinatorial feature of the left/right (intro/elim)
rule pairs, namely that a complete set of premises can be refuted. The
Curry-Howard isomorphism between derivations and terms in a lambda
calculus, and correspondingly between proof reductions and
$\beta$-reductions in the lambda calculus, requires, and is made
possible by, the same property.  The general perspective taken here
shows how the Curry-Howard correspondence pairs
discharge of assumptions with abstraction, and proof substitutions with
reduction, in the following sense: the constructor and destructor
terms corresponding to intro and elim rules require the abstraction of
a variable every time an assumption is discharged.  The $\lif$-intro
rule and simple $\lambda$-abstraction appear as special cases where
discharge in a single premise corresponds to abstraction of a single
variable, and the constructor function symbol can essentially be left
out.

This perspective is of course not the only possible one. It
generalizes the various proof-theoretic frameworks in one direction,
namely sets of connectives and rules other than the usual set of
$\lnot$, $\land$, $\lor$, $\lif$. It is applicable to any (set of)
connectives that has a truth-table semantics (in the classical case).
The intuitionistic case then arises in the same way that $\LJ$ arises
from $\LK$ by restriction to a single formula in the succedent. Other
approaches are more appropriate if one wants to avoid a semantics as a
starting point, or at least not assume that the connectives considered
have a classical semantics. \cite{Schroeder-Heister1984} has
generalized natural deduction directly by allowing not just formulas,
but rules as assumptions, and considers natural deduction rules for
arbitrary connectives in this framework. Another general perspective
is taken by \cite{SambinBattilottiFaggian2000}, where the
generalization covers more kinds of calculi (e.g., substructural
systems). They take as their starting point the introduction (right)
rules for a connective and derive the elimination (left) rules from it
by a process of ``reflection.'' The connection between these two
approaches is investigated in \cite{Schroeder-Heister2013}.

In the rest of the paper we review the construction of general sequent
calculi (\S\ref{sec:rules}) and show how the usual rules of~$\LK$
arise as the result of splitting of rules, to ensure at most formula
occurs in the succedent of any premise  (\S\ref{sec:splitting}). From
these sequent calculi we obtain multi-conclusion natural deduction
systems with general elimination rules (\S\ref{sec:mcnd}). The usual
natural deduction rules arise by specializing these general
elimination rules (\S\ref{sec:specialize}). We show that the cut
elimination theorem holds for the sequent calculi so constructed
(\S\ref{sec:cutel}) and that normalization holds for the
multi-conclusion natural deduction systems
(\S\ref{sec:normalization}). The process of splitting rules guarantees
that for any connective there are candidate rules for a
single-conclusion sequent calculus, which relates to the
multi-conclusion system as intuitionistic sequent calculus~$\LJ$
relates to classical~$\LK$. We describe such ``intuitionistic''
single-conclusion sequent calculi (\S\ref{sec:sclk}) and explain how
to obtain a classical single-conclusion system equivalent to the
multi-conclusion system (\S\ref{sec:classical}). From
single-conclusion sequent calculi we can in turn obtain
single-conclusion natural deduction systems (\S\ref{sec:scnd}). The
rules of single-conclusion natural deduction correspond to
constructors and destructors for a generalized lambda calculus, of
which the usual typed lambda calculus is again a special case
(\S\ref{sec:ch}). We describe how the general construction of rules
also extends to Parigot's system of free deduction (\S\ref{sec:fd}),
recently rediscovered by Milne as natural deduction with  general
elimination and general introduction rules. Finally, we sketch how to
extend the approach to quantifiers (\S\ref{sec:quantifiers}).

\section{Complete sets of sequent calculus rules}
\label{sec:rules}

We recall some terminology: A \emph{literal} is an atomic or a negated
atomic formula. A disjunction of literals is also called a
\emph{clause} and is often often written simply as a set.  Thus, the
disjunction $\lnot A_1 \lor A_2$ may also be written as $\{\lnot A_1,
A_2\}$.  Satisfaction conditions for clauses under a valuation~$\M$
are just those of the corresponding disjunction, i.e., if $C = \{L_1,
\dots, L_n\}$ is a clause, $\M \models C$ iff $\M \models L_i$ for at
least one~$L_i$.  A set of clauses is satisfied in~$\M$ iff each
clause in it is.  Thus, a set of clauses corresponds to a conjunctive
normal form, i.e., a conjunction of disjunctions of literals.

Now consider a truth-functional connective~$\opa$.  To say that $\opa$
is truth-functional is to say that its semantics is given by a truth
function $\tilde\opa \colon \{\top, \bot\}^n \to \{\top, \bot\}$, and
that the truth conditions for a formula $A \equiv \opa(A_1, \dots,
A_n)$ are given by:
\[
\M \models \opa(A_1, \dots, A_n) \text{ iff } \tilde\opa(v_1, \dots,
v_n) = \top,
\]
where $v_i = \top$ if $\M \models A_i$ and $= \bot$ otherwise.

As is well known, every truth-functional connective $\opa(A_1, \dots,
A_n)$ can be expressed by a conjunctive normal form in the $A_i$,
i.e., a conjunction of clauses in the~$A_i$, and the same is true for
its negation.  In other words, for every truth function~$\tilde\opa$,
there is a set of clauses~$\C(\opa)^+$ which is satisfied in~$\M$ iff
$\opa(A_1, \dots, A_n)$ is, and a set of clauses $\C(\opa)^-$ which is
not satisfied iff $\opa(A_1, \dots, A_n)$ is not.  $\C(\opa)^+$ and
$\C(\opa)^-$ are of course not unique.

The \emph{Kowalski notation} for a clause~$C = \{\lnot A_1, \dots
\lnot A_k, B_1, \dots, B_l\}$ ($A_i$, $B_j$ atomic) is an expression
of the form $A_1, \dots, A_k \fCenter B_1, \dots, B_l$.  We can now
establish a correspondence between truth-functional
connectives~$\opa$, their associated clause sets $\C(\opa)^+$ and
$\C(\opa)^-$, and sequent calculus rules for them.  If $\C$ is a
clause set, then the corresponding set of premises is the set of
sequents obtained from the Kowalski notations of the clauses in~$\C$
by adding schematic variables for formula sequences $\Gamma$, $\Delta$
in the antecedent and succedent.  By a slight abuse of notation, we
use the same meta-variables~$A_i$ for the schematic formulas in the
sequent on the one hand, and the atomic formulas in the clause set on
the other.  Each premise has the form $\Pi, \Gamma \fCenter \Delta,
\Lambda$ where the formulas in $\Pi$ are the $A_i$ occurring
negatively in the respective clause, and those in $\Lambda$ are the
$A_i$ occurring positively.  The $\Left\opa$ rule is the rule with the
premise set corresponding to $\C(\opa)^-$ and the conclusion
$\opa(A_1, \dots, A_n), \Gamma \fCenter \Delta$. The $\Right\opa$ rule
has a premise set corresponding to $\C(\opa)^+$ and conclusion is
$\Gamma \fCenter \Delta, \opa(A_1, \dots, A_n)$. Rules have the
general forms
\[
  \Axiom$\dots \quad \Pi_i, \Gamma \fCenter \Delta, \Lambda_i \quad \dots$
  \RightLabel{$\Right\opa$}
  \UnaryInf$\Gamma \fCenter \Delta, \opa(\vec A)$
  \DisplayProof
  \quad\text{and}\quad
\Axiom$\dots \quad \Pi_i, \Gamma \fCenter \Delta, \Lambda_i \quad \dots$
\RightLabel{$\Left\opa$}
\UnaryInf$\opa(\vec A), \Gamma \fCenter \Delta$
\DisplayProof
\]
where the formulas in $\Pi_i$ and $\Lambda_i$ in $\Right\opa$ correspond to
$\C(\opa)^+$ and those in $\Left\opa$ to $\C(\opa)^-$.  We call the
occurrence of $\opa(\vec A)$ in the conclusion of such a rule the
\emph{principal} formula, the formulas in $\Pi_i$ and $\Lambda_i$ (or
$\Pi_i'$ and $\Lambda_i'$) the \emph{auxiliary} formulas, and those in
$\Gamma$ and $\Delta$ the \emph{side} formulas.

\begin{table}
  \caption{Clause sets and rules for the usual connectives}
\begin{minipage}{\textwidth}
\[
\begin{array}{clc}
\hline\hline
  \text{connective} & \text{\textsc{cnf}s} \,/\, \C(\opa)^+,
  \C(\opa)^- & \text{rules}\\ \hline\\ 
\land & 
\begin{array}{l}
A \land B \\
\{\fCenter A; \fCenter B\}
\end{array} & 
\Axiom$\Gamma \fCenter \Delta, A$
\Axiom$\Gamma \fCenter \Delta, B$
\RightLabel{$\Right\land$}
\BinaryInf$\Gamma \fCenter \Delta, A \land B$
\DisplayProof
 \\[2ex] &
\begin{array}{l}
\lnot A \lor \lnot B \\
\{A, B \fCenter\}
\end{array} & 
\Axiom$A, B, \Gamma \fCenter \Delta$
\RightLabel{$\Left\land$}
\UnaryInf$A \land B, \Gamma \fCenter \Delta$
\DisplayProof \\[2ex]
\lor & 
\begin{array}{l}
A \lor B \\
\{\fCenter A, B\}
\end{array}
 & 
\Axiom$\Gamma \fCenter \Delta, A, B$
\RightLabel{$\Right\lor$}
\UnaryInf$\Gamma \fCenter \Delta, A \lor B$
\DisplayProof
    \\[2ex] &
\begin{array}{l}
\lnot A \land \lnot B \\
 \{A \fCenter; B \fCenter\}
 \end{array}
  &
\Axiom$A, \Gamma \fCenter \Delta$
\Axiom$B, \Gamma \fCenter \Delta$
\RightLabel{$\Left\lor$}
\BinaryInf$A \lor B, \Gamma \fCenter \Delta$
\DisplayProof\\
\lif & 
\begin{array}{l}
\lnot A \lor B \\
\{A \fCenter B\} 
\end{array}
& 
\Axiom$A, \Gamma \fCenter \Delta, B$
\RightLabel{$\Right\lif$}
\UnaryInf$\Gamma \fCenter \Delta, A \lif B$
\DisplayProof
 \\[2ex] 
& 
\begin{array}{l}
A \land \lnot B \\
 \{\fCenter A; B \fCenter\} 
 \end{array}
 & 
\Axiom$\Gamma \fCenter \Delta, A$
\Axiom$B, \Gamma \fCenter \Delta$
\RightLabel{$\Left\lif$}
\BinaryInf$A \lif B, \Gamma \fCenter \Delta$
\DisplayProof\\[2ex]
\hline\hline
\end{array}
\]
\end{minipage}
\end{table}

\begin{table}
  \caption{Clause sets and rules for some unusual connectives}
  \label{unusual}
$
\begin{array}{clc}
\hline\hline
\text{connective} & \text{\textsc{cnf}s} \,/\, \C(\opa)^+,
  \C(\opa)^- & \text{rule}\\ \hline\\ 
A \nif B & 
\begin{array}{l}
A \land \lnot B \\
 \{\fCenter A; B \fCenter\}
 \end{array}
 & 
\Axiom$\Gamma \fCenter \Delta, A$
\Axiom$B, \Gamma \fCenter \Delta$
\RightLabel{$\Right\nif$}
\BinaryInf$\Gamma \fCenter \Delta, A \nif B$
\DisplayProof\\[2ex]
\text{(exclusion)} & 
\begin{array}{l}
\lnot A \lor B \\
  \{A \fCenter B\}
  \end{array}
 & 
\Axiom$A, \Gamma \fCenter \Delta, B$
\RightLabel{$\Left\nif$}
\UnaryInf$A \nif B, \Gamma \fCenter \Delta$
\DisplayProof \\[2ex] 
A \mid B & 
\begin{array}{l}
\lnot A \lor \lnot B \\
 \{A, B \fCenter\}
 \end{array}
 & 
\Axiom$  A, B, \Gamma \fCenter \Delta$
\RightLabel{$\Right\mid$}
\UnaryInf$\Gamma \fCenter \Delta, A \mid B$
\DisplayProof \\[2ex]
\text{(nand)} & 
\begin{array}{l}
A \land B \\
  \{\fCenter A; \fCenter B\} 
  \end{array}
& 
\Axiom$\Gamma \fCenter \Delta, A$
\Axiom$\Gamma \fCenter \Delta, B$ 
\RightLabel{$\Left\mid$}
\BinaryInf$A \mid B, \Gamma \fCenter \Delta$
\DisplayProof \\[2ex]
A \nor B & 
\begin{array}{l}
\lnot A \land \lnot B \\
 \{A \fCenter; B \fCenter\} 
\end{array}
& 
\Axiom$A, \Gamma \fCenter \Delta$
\Axiom$B, \Gamma \fCenter \Delta$
\RightLabel{$\Right\nor$}
\BinaryInf$\Gamma \fCenter \Delta, A \nor B$
\DisplayProof \\[2ex]
\text{(nor)} & 
\begin{array}{l}
A \lor B \\
  \{\fCenter A, B\} 
\end{array}
& 
\Axiom$\Gamma \fCenter \Delta, A, B$
\RightLabel{$\Left\nor$}
\UnaryInf$A \nor B, \Gamma \fCenter \Delta$
\DisplayProof \\[2ex]
A \xor B & 
\begin{array}{l}
(A \lor B) \land (\lnot A \lor \lnot B) \\
\{\fCenter A, B; A, B \fCenter\} 
\end{array}
& 
  \Axiom$\Gamma \fCenter \Delta, A, B$
  \Axiom$A, B, \Gamma \fCenter \Delta$
\RightLabel{$\Right\xor$}
\BinaryInf$\Gamma \fCenter \Delta, A \xor B$
\DisplayProof \\[2ex]
\text{(xor)} & 
\begin{array}{l}
(A \lor \lnot B) \land (\lnot A \lor B) \\
\{A \fCenter B; B \fCenter A\}
\end{array}
 & 
\Axiom$A, \Gamma \fCenter \Delta, B$
\Axiom$B, \Gamma \fCenter \Delta, A$
\RightLabel{$\Left\xor$}
\BinaryInf$A \xor B, \Gamma \fCenter \Delta$
\DisplayProof \\[2ex]
A \lif B / C & 
\begin{array}{l}
(\lnot A \lor B) \land (A \lor C) \\
\{A \fCenter, B; \fCenter A, C\}
\end{array}
 & 
\Axiom$A, \Gamma \fCenter \Delta, B$
\Axiom$\Gamma \fCenter \Delta, A, C$
\RightLabel{$\Right{\ite}$}
\BinaryInf$\Gamma \fCenter \Delta, A \lif B/C$
\DisplayProof\\[2ex]
\text{(if then else)} & 
\begin{array}{l}
(\lnot A \lor \lnot B) \land (A \lor \lnot C) \\
\{A, B \fCenter; C \fCenter A\} 
\end{array}
& 
\Axiom$A, B, \Gamma \fCenter \Delta$
\Axiom$C, \Gamma \fCenter \Delta, A$
\RightLabel{$\Left{\ite}$}
\BinaryInf$A \lif B/C, \Gamma \fCenter \Delta$
\DisplayProof\\[2ex]
\hline\hline
\end{array}
$
\end{table}

Suppose we have a set $X$ of connectives $\opa$ with associated $\Left\opa$ and $\Right\opa$ rules.  The calculus consisting of these rules, the
usual structural rules (weakening, contraction, exchange, cut), and
initial sequents $A \fCenter A$ is the classical sequent
calculus~$\L X$ for those connectives and rules.

\begin{prop}
$\L X$ with rule pairs for each connective as defined above is
  sound and complete.
\end{prop}

\begin{proof}
  Soundness by a standard inductive proof, completeness by the usual
  construction of a countermodel from a failed search for a cut-free
  proof. See Theorem 3.2 of \cite{BaazFermullerZach1994} for details.
  (Note that this proof yields completeness without the cut rule.)
\end{proof}

When comparing the systems $\L X$ with natural deduction systems, it
will be useful to start from a slight variant of $\L X$ in which the
side formulas are not required to be identical in the premises.
We replace an $\L X$ rule 
\[
\Axiom$\Pi_1, \Gamma \fCenter \Delta, \Lambda_1$
\AxiomC{\dots}
\Axiom$\Pi_n, \Gamma \fCenter \Delta, \Lambda_n$
\RightLabel{$\Right\opa$}
\TrinaryInf$\Gamma \fCenter \Delta, \opa(\vec A)$
\DisplayProof
\]
by
\[
\Axiom$\Pi_1, \Gamma_1 \fCenter \Delta_1, \Lambda_1$
\AxiomC{$\dots$}
\Axiom$\Pi_n, \Gamma_n \fCenter \Delta_n, \Lambda_n$
\RightLabel{$\Right\opa^c$}
\TrinaryInf$\Gamma_1, \dots, \Gamma_n \fCenter 
  \Delta_1, \dots, \Delta_n, \opa(\vec A)$
\DisplayProof
\]
to obtain the sequent calculus with independent contexts~$\Lc X$.

\begin{prop}
$\L X$ and $\Lc X$ are equivalent.
\end{prop}

\begin{proof}
A proof in $\L X$ can be translated into one in $\Lc X$ by adding
contractions (and exchanges), one in $\Lc X$ can be translated into a
proof in $\L X$ by adding weakenings (and exchanges).
\end{proof}

\section{Splitting rules}
\label{sec:splitting}

The sequent calculi usually considered differ from the calculi~$\L X$
as constructed above in that some of the rules are ``split.'' This is
the case in particular for the $\Right\lor$ rule in~$\LK$.  In this case,
we replace the single rule with premise clause $\fCenter A, B$ by two
rules, each with a single premise $\fCenter A$ or $\fCenter B$.  The
availability of split rules is crucial if one wants to consider a
variant of the sequent calculus with one side restricted to at most
one formula, and when considering natural deduction systems, where the
conclusion is also likewise restricted to a single formula. In such
systems, we need rules where the premises satisfy the restriction as
well.

Suppose we have a rule for $\opa$ with premise set corresponding to a
set of clauses~$\C$, and let $C \in \C$ be the clause corresponding to
a single premise of the rule. We can partition the clause~$C$ into two
parts: $C = C_1 \cup C_2$, where $C_2 = \{L_1, \dots, L_k\}$ and
consider the clause sets $\C_i = (\C\setminus \{C\}) \cup \{C_1 \cup\{
L_i\}\}$.  We can then consider the variant calculus~$\L X^*$ with the
rule corresponding to $\C$ replaced by the rules corresponding to all
the ~$\C_i$.

\begin{prop}
$\L X$ and $\L X^*$ are equivalent.
\end{prop}

\begin{proof}
We just have to show that an inference based on the rule for $\C$ can
be simulated by the rules for $\C_i$ and conversely.  An inference
using $\C_i$ can be simulated by an inference based on~$\C$ by
weakening the premise corresponding to $C_1 \cup \{L_i\}$ with the
rest of~$C_2$.  In the other direction, an inference based on $\C$
can be simulated by repeated application of $\C_i$ and contracting
multiple copies of $\opa(A_1, \dots, A_n)$ in the resulting conclusion
sequent.
\end{proof}

For instance, in the case of $\Right\lor$, the constructed rule with
premise clause $C = {} \fCenter A, B$ can be replaced with the two $\Right\lor$
rules with premises $\fCenter A$ and $\fCenter B$ (here $C_1 = \emptyset$
and $C_2 = \{A, B\}$).  An inference using $\Right\lor_1$ with premise
$\Gamma \fCenter \Delta, A$ can be replaced by
\[
\Axiom$\Gamma \fCenter \Delta, A$
\RightLabel{$\Right\Weak$}
\UnaryInf$\Gamma \fCenter \Delta, A, B$
\RightLabel{$\Right\lor$}
\UnaryInf$\Gamma \fCenter \Delta, A \lor B$
\DisplayProof
\]
An inference using $\Right\lor$ can be replaced by a proof segment using
$\Right\lor_1$ and $\Right\lor_2$:
\[
\Axiom$\Gamma \fCenter \Delta, A, B$
\RightLabel{$\Right\lor_2$}
\UnaryInf$\Gamma \fCenter \Delta, A, A \lor B$
\RightLabel{$\Right\Exch$}
\UnaryInf$\Gamma \fCenter \Delta, A \lor B, A$
\RightLabel{$\Right\lor_1$}
\UnaryInf$\Gamma \fCenter \Delta, A \lor B, A \lor B$
\RightLabel{$\Right\Contr$}
\UnaryInf$\Gamma \fCenter \Delta, A \lor B$
\DisplayProof
\]

We thus have:

\begin{prop}
If $\L X^*$ is obtained from $\L X$ by splitting a rule, it is sound
and complete iff $\L X$ is.
\end{prop}

\begin{cor}\label{intuitionistic-rules}
If $\L X$ is sound and complete, there is a sound and complete $\L
X^*$ in which every premise of every rule has at most one auxiliary
formula on the right.
\end{cor}

\begin{proof}
By successively replacing every rule in which some premise contains
more than one auxiliary formula on the right by the split rules where $C_2$
contains all the positive clauses of the corresponding premise
clause.
\end{proof}

Note that this result does not establish that the calculus obtained
from $\L X^*$ by restricting all sequents to at most one formula on
the right proves the same (restricted) sequents as $\L X$.  It only
shows that for every $\L X$ there is at least a candidate calculus $\L
X^*$ where such a restriction would be possible as far as the logical
rules are concerned. The resulting calculus may have different
provable sequents. The example of $\LK$ and $\LJ$ illustrates this.

The set of rules obtained by splitting a rule of $\L X$ more than once
may result in redundant rules, e.g., those where a premise contains
an auxiliary formula $A$ on the left and another premise contains it
on the right. For instance, consider the $\Right\xor$ rule from
Table~\ref{unusual}. Fully splitting the rule would result in four
rules:
\[
\begin{array}{cc}
\Axiom$\Gamma \fCenter \Delta, A$
\Axiom$B, \Gamma \fCenter \Delta$
\RightLabel{$\Right\xor_1$}
\BinaryInf$\Gamma \fCenter \Delta, A \xor B$
\DisplayProof &
\Axiom$\Gamma \fCenter \Delta, B$
\Axiom$A, \Gamma \fCenter \Delta$ 
\RightLabel{$\Right\xor_2$}
\BinaryInf$\Gamma \fCenter \Delta, A \xor B$
\DisplayProof
\\[2ex]
\Axiom$\Gamma \fCenter \Delta, A$
\Axiom$A, \Gamma \fCenter \Delta$
\RightLabel{$\Right\xor_3$}
\BinaryInf$\Gamma \fCenter \Delta, A \xor B$
\DisplayProof
 &
\Axiom$\Gamma \fCenter \Delta, B$
\Axiom$B, \Gamma \fCenter \Delta$
\RightLabel{$\Right\xor_4$}
\BinaryInf$\Gamma \fCenter \Delta, A \xor B$
\DisplayProof
\end{array}
\]
The two bottom rules are superfluous: the un-split rule
can be simulated using the other two (and contractions):
\[
\Axiom$\Gamma \fCenter \Delta, A, B$
\Axiom$A \fCenter A$
\doubleLine
\UnaryInf$A, \Gamma \fCenter \Delta, A$
\RightLabel{$\Right\xor_2$}
\BinaryInf$\Gamma \fCenter \Delta, A, A \xor B$
\doubleLine
\UnaryInf$\Gamma \fCenter \Delta, A \xor B, A$
    \Axiom$B \fCenter B$
    \doubleLine
    \UnaryInf$B, \Gamma \fCenter \Delta, B$
    \Axiom$A, B, \Gamma \fCenter \Delta$
    \RightLabel{$\Right\xor_2$}
    \BinaryInf$B, \Gamma \fCenter \Delta, A \xor B$
\RightLabel{$\Right\xor_1$}
\BinaryInf$\Gamma \fCenter \Delta, A \xor B, A \xor B$
\RightLabel{$\Right\Contr$}
\UnaryInf$\Gamma \fCenter \Delta, A \xor B$
\DisplayProof\]

\begin{table}
  \caption{Split rules for some unusual connectives}
\[
\begin{array}{ccc}
\hline\hline
\text{connective} & \text{rules} \\ \hline\\ 
A \nif B & 
\Axiom$A, \Gamma \fCenter \Delta$
\RightLabel{$\Left\nif_1$}
\UnaryInf$A \nif B, \Gamma \fCenter \Delta$
\DisplayProof
& 
\Axiom$\Gamma \fCenter \Delta, B$
\RightLabel{$\Left\nif_2$}
\UnaryInf$A \nif B, \Gamma \fCenter \Delta$
\DisplayProof \\[2ex] 
A \mid B 
&  
\Axiom$A, \Gamma \fCenter \Delta$
\RightLabel{$\Right\mid_1$}
\UnaryInf$\Gamma \fCenter \Delta, A \mid B$
\DisplayProof
& 
\Axiom$B, \Gamma \fCenter \Delta$
\RightLabel{$\Right\mid_2$}
\UnaryInf$\Gamma \fCenter \Delta, A \mid B$
\DisplayProof \\[2ex]
A \nor B 
& 
\Axiom$    \Gamma \fCenter \Delta, A$
\RightLabel{$\Left\nor_1$}
\UnaryInf$A \nor B, \Gamma \fCenter \Delta$
\DisplayProof 
& 
\Axiom$\Gamma \fCenter \Delta, B$
\RightLabel{$\Left\nor_2$}
\UnaryInf$A \nor B, \Gamma \fCenter \Delta$
\DisplayProof \\[2ex]
A \xor B 
& 
\Axiom$\Gamma \fCenter \Delta, A$
\Axiom$B, \Gamma \fCenter \Delta$
\RightLabel{$\Right\xor_1$}
\BinaryInf$\Gamma \fCenter \Delta, A \xor B$
\DisplayProof
& 
\Axiom$\Gamma \fCenter \Delta, B$
\Axiom$A, \Gamma \fCenter \Delta$
\RightLabel{$\Right\xor_2$}
\BinaryInf$\Gamma \fCenter \Delta, A \xor B$
\DisplayProof\\[2ex]
& 
\Axiom$A, \Gamma \fCenter \Delta$
\Axiom$B, \Gamma \fCenter \Delta$
\RightLabel{$\Left\xor_1$}
\BinaryInf$A \xor B, \Gamma \fCenter \Delta$
\DisplayProof
& 
\Axiom$\Gamma \fCenter \Delta, B$
\Axiom$\Gamma \fCenter \Delta, A$
\RightLabel{$\Left\xor_2$}
\BinaryInf$A \xor B, \Gamma \fCenter \Delta$
\DisplayProof \\[2ex]
A \lif B / C 
& 
\Axiom$A, \Gamma \fCenter \Delta, B$
\Axiom$\Gamma \fCenter \Delta, A$
\RightLabel{$\Right\ite_1$}
\BinaryInf$\Gamma \fCenter \Delta, A \lif B/C$
\DisplayProof
& 
\Axiom$A, \Gamma \fCenter \Delta, B$
\Axiom$ \Gamma \fCenter \Delta, C$
\RightLabel{$\Right\ite_2$}
\BinaryInf$\Gamma \fCenter \Delta, A \lif B/C$
\DisplayProof\\[2ex]
& 
\Axiom$A, B, \Gamma \fCenter \Delta$
\Axiom$\Gamma \fCenter \Delta, A$
\RightLabel{$\Left\ite_1$}
\BinaryInf$A \lif B/C, \Gamma \fCenter \Delta$
\DisplayProof
& 
\Axiom$A, B, \Gamma \fCenter \Delta$
\Axiom$C, \Gamma \fCenter \Delta$
\RightLabel{$\Left\ite_2$}
\BinaryInf$A \lif B/C, \Gamma \fCenter \Delta$
\DisplayProof
\\[2ex]
\hline\hline
\end{array}
\]
\end{table}

\begin{cor}
  If $\L X$ is sound and complete, there is a sound and complete $\L
  X^{**}$ in which every premise of every rule has at most one auxiliary
  formula.
\end{cor}

$\L X^{**}$ is a calculus with ``fully split'' rules.  In this case,
each premise of a rule corresponds to a single literal (positive if
the auxiliary formula appears on the right, negative if on the left).
A rule then corresponds to a conjunction of literals, and a set of
$\Right\opa$ rules to a \textsc{dnf} of $\opa(A_1, \dots, A_n)$ in the atomic
formulas~$A_i$.  Starting from a single $\Right\opa$ rule corresponding to
a \textsc{cnf}, the fully split set of rules corresponds to converting the \textsc{cnf}
to a \textsc{dnf} by distributing $\land$ over $\lor$.  The fully split rules
have the interesting property that when applied to initial sequents
they derive the sequent $\Gamma \fCenter \Delta, \opa(A_1, \dots, A_n)$
where $\Gamma$ contains $A_i$ iff $A_i$ appears as an auxiliary formula
on the right, and $\Delta$ contains $A_i$ iff $A_i$ appears as an
auxiliary formula on the left. If $\lnot$ is present, we even get a
derivation, using only $\Left\lnot$ and $\Right\opa$, of $\Pi \fCenter
\opa(\vec A)$ (which moreover does not contain more than one formula
on the right), where $\Pi = \Gamma, \lnot \Delta$. The set of all
$\Pi$ so obtained exactly describes the truth value assignments under
which $\opa(\vec A)$ is true.

Suppose now we have a sequent calculus with the usual structural rules
but only right rules for a connective $\opa$.  We can determine sound
and complete left rules for~$\opa$ by reverse engineering the rules
given.  Each $\Right\opa$ rule determines a set of clauses, which is
equivalent to a conjunctive normal form for $\opa(A_1, \dots, A_n)$ in
the arguments $A_i$, \dots, $A_n$.  If we have more than one
$\Right\opa$ rule, consider the disjunction of the corresponding
\textsc{cnf}s. Its negation also has a \textsc{cnf}, which yields an
$\Left\opa$ rule which together with the $\Right\opa$ rules is sound
and complete, as is the calculus resulting from the given $\Right\opa$
rules together with any rules obtained from $\Left\opa$ by splitting.

\section{Multi-conclusion natural deduction rules}
\label{sec:mcnd}

A sequent calculus $\L X$ is straightforwardly and systematically
related to a sequence-conclusion ``sequent-style'' natural deduction
calculus $\Nms X$ as follows. In $\Nms X$, the antecedent of a sequent
$\Gamma \fCenter \Delta$ is a multiset, not a sequence as in $\L X$. The
axioms of $\Nms X$ are initial sequents $A \fCenter A$.  Every $\Right\opa$
rule is also a rule in $\Nms X$, an \emph{introduction rule} for
$\opa$.  Every $\Left\opa$ rule corresponds to a \emph{general
  elimination rule} with the same premises as
$\Left\opa$ (the minor premises), an additional premise $\Gamma
\fCenter \Delta,
\opa(A_1, \ldots, A_n)$ (the major premise) and the conclusion $\Gamma
\fCenter \Delta$, where $\Gamma$ and $\Delta$ are the same schematic
context formulas as in the premises. As an example, consider the
conditional, $\lif$.  \textsc{cnf}s for $A \lif B$ and $\lnot(A \lif B)$ are
\begin{align*}
A \lif B & \text{ iff } \lnot A \lor B\\
\lnot(A \lif B) & \text{ iff } A \land \lnot B
\end{align*}
These correspond to the introduction and elimination rules
\[
\Axiom$A, \Gamma \fCenter \Delta, B$
\RightLabel{$\Intro\lif$}
\UnaryInf$\Gamma \fCenter \Delta, A \lif B$
\DisplayProof
\qquad
\Axiom$\Gamma \fCenter \Delta, A \lif B$
\Axiom$\Gamma \fCenter \Delta, A$
\Axiom$B, \Gamma \fCenter \Delta$
\RightLabel{$\Elim\lif$}
\TrinaryInf$\Gamma \fCenter \Delta$
\DisplayProof
\]
The formulas in the antecedent of a sequent in an $\Nms X$ derivation
are called \emph{open assumptions}.  If a schematic variable $A$ occurs
in the antecedent of a premise in an elimination rule, it does not
occur in the antecedent of the conclusion: the open assumption $A$ is
then said to be \emph{discharged} by the rule.

Any proof in $\L X$ can be translated into a proof in $\Nms X$.  A
$\Left\opa$ inference translates into a $\Elim\opa$ inference with major
premise $\opa(A_1, \dots, A_n), \Gamma \fCenter \Delta, \opa(A_1, \dots,
A_n)$.  This premise itself can be derived from $\opa(A_1, \dots, A_n)
\fCenter \opa(A_1, \dots, A_n)$ by weakening and exchanges alone.
Conversely, any proof in $\Nms X$ can be translated into an $\L X$
proof.  An $\Elim\opa$ inference is translated into a $\Left\opa$ inference
(with the minor premises as premises of the rule), followed by a cut
with the proof ending in the major premise.

The sequent-style natural deduction calculus~$\Nms X$ corresponds to a
natural deduction system in which proofs are trees of sequences of
formulas.\footnote{Such a system was first introduced by
\cite{Kutschera1962}, for a language with Sheffer stroke and universal
quantifier as the only primitives; he did not investigate normal
forms. \cite{Boricic1985} gave a system for the usual primitives and
provided explicit translations from and to the sequent calculus and a
normal form theorem. \cite{BaazFermullerZach1993b} independently
obtained generalized results of the same type for arbitrary $n$-valued
connectives and quantifiers. \cite{Cellucci1992} proved normalization
for a classical system including the $\varepsilon$-operator.
\cite{Parigot1992a} gave a system in which conclusions are sets of
``named'' formulas, and proved strong normalization using the
$\lambda\mu$-calculus. The system is related to the
multiple-conclusion calculi of \cite{ShoesmithSmiley1978}, although
their proofs have a different structure. In their systems, proofs are
graphs of formulas, and a rule may have more than one conclusion
formula, resulting in proofs that are not trees. Our system has trees
of sequents with more than one formula in the succedent, but naturally
yields a system of trees of sequences or multi-sets of formulas, and
every rule has a single sequence or multi-set of formulas as
conclusion.} Whenever there is a derivation of $\Delta$ from
undischarged assumptions $\Gamma$ there is a derivation of $\Gamma
\fCenter \Delta$ in $\Nms X$. Conversely, if $\Nms X$ proves $\Gamma
\fCenter \Delta$ then there is a derivation in $\Nm X$ of $\Gamma'
\fCenter \Delta$ with $\Gamma' \subseteq \Gamma$. (Note that in
natural deduction there is no equivalent to the left weakening rule,
hence in general we can only get $\Gamma' \subseteq \Gamma$).

%\section{Multi-conclusion Natural Deduction with Contexts}

In the context of natural deductions (especially when considering the
Curry-Howard isomorphism) it is often necessary and helpful to
consider a version of natural deduction in which one has more control
over which elimination inferences discharge which assumptions, and to
have book-keeping information for this in the derivations themselves.
A corresponding sequent-style natural deduction calculus then will
have \emph{labels} on the formulas occurring in the antecedent (the
assumptions), and the antecedent is now considered a set, not a
multi-set, of labelled formulas. The antecedent of a sequent is also
often called the \emph{context}. The $\Left\Weak$ and $\Left\Contr$
rules are removed. Since the contexts are sets, contraction on the
left is implicit. For simplicity we will also consider the consequent
to be a \emph{multi}-set of formulas (i.e., exchanges are implicit and
the $\Right\Exch$ rule is removed). In order to make the comparison
with type systems easier, we will use variables~$x$ as these labels,
and write a labelled assumption formula as $x \colon A$.  An initial
sequent then is a sequent of the form
\[
x \colon A \fCenter A
\]
We replace the introduction and elimination rules by rules in which
the side formulas are not required to be shared.  In other words,
we replace a rule
\[
\Axiom$\dots \quad \Pi_i, \Gamma \fCenter \Delta, \Lambda_ i \quad \dots$ 
\RightLabel{$\Intro\opa$}
\UnaryInf$\Gamma \fCenter \Delta$
\DisplayProof
\quad\text{by}\quad
\Axiom$\dots \quad \Pi_i, \Gamma_i \fCenter \Delta_i, \Lambda_i \quad \dots$
\RightLabel{$\Intro\opa$}
\UnaryInf$\Gamma_1, \dots, \Gamma_n \fCenter \Delta_1, \dots, \Delta_n$
\DisplayProof
\]
Since the $\Left\Weak$ rule is removed, we must allow \emph{vacuous discharge}
of assumptions: a rule application is correct even if the
discharged assumptions do not all appear in the premises.
In other words, an application of a logical rule
\[
\Axiom$\Gamma_0 \fCenter \Delta, \opa(\vec A)$
\AxiomC{$\Pi_1, \Gamma_1 \fCenter \Delta_1, \Lambda_1 \quad 
  \dots \quad
  \Pi_n, \Gamma_n \fCenter \Delta_n, \Lambda_n$}
\RightLabel{$\Elim\opa \colon l$}
\BinaryInf$\Gamma_0, \dots, \Gamma_n \fCenter 
  \Delta_1, \dots, \Delta_n$
  \DisplayProof
\]
may discharge any number of assumptions mentioned in the schematic
form of the rule, including zero.  The labels of assumptions
discharged in an application are listed with the rule.  Formulas
instantiating the same schematic variable must have the same label.
For instance, the $\Elim\lif$ rule now becomes:
\[
\Axiom$\Gamma_0 \fCenter \Delta_0, A \lif B$
\Axiom$\Gamma_1 \fCenter \Delta_1, A$
\Axiom$x \colon B, \Gamma_2 \fCenter \Delta_2$
\RightLabel{$\Elim\lif \colon x$}
\TrinaryInf$\Gamma_0, \Gamma_1, \Gamma_2 \fCenter 
  \Delta_0, \Delta_1, \Delta_2$
  \DisplayProof
\]

A correct derivation is a tree of sequents starting in initial
sequents in which every inference is correct according to the new
definition.  We denote the new system $\Nmsl X$.

\begin{table}
  \caption{Multi-conclusion natural deduction rules for some unusual connectives}
\[
\begin{array}{cllc}
\hline\hline
\text{connective} & \text{intro/elim rule} \\ \hline\\ 
A \nif B & 
\Axiom$\Gamma_1 \fCenter \Delta_1, A$
\Axiom$B, \Gamma_2 \fCenter \Delta_2$
\RightLabel{$\Intro\nif$}
\BinaryInf$\Gamma_1, \Gamma_2 \fCenter \Delta_1, \Delta_2, A \nif B$
\DisplayProof\\[2ex]
& 
\Axiom$\Gamma_0 \fCenter \Delta_0, A \nif B$
\Axiom$A, \Gamma_1 \fCenter \Delta_1, B$
\RightLabel{$\Elim\nif$}
\BinaryInf$\Gamma_0, \Gamma_1 \fCenter \Delta_0, \Delta_1$
\DisplayProof \\[2ex] 
A \mid B &  
\Axiom$A, B, \Gamma_1 \fCenter \Delta_1$
\RightLabel{$\Intro\mid$}
\UnaryInf$\Gamma_1 \fCenter \Delta_1, A \mid B$
\DisplayProof \\[2ex]
& 
\Axiom$\Gamma_0 \fCenter \Delta_0, A \mid B$
\Axiom$\Gamma_1 \fCenter \Delta_1, A$ 
\Axiom$\Gamma_2 \fCenter \Delta_2, B$ 
\RightLabel{$\Elim\mid$}
\TrinaryInf$\Gamma_0, \Gamma_1, \Gamma_2 \fCenter \Delta_0, \Delta_1, \Delta_2$
\DisplayProof\\[2ex]
A \nor B &  
\Axiom$A, \Gamma_1 \fCenter \Delta_1$
\Axiom$B, \Gamma_2 \fCenter \Delta_2$
\RightLabel{$\Intro\nor$}
\BinaryInf$\Gamma_1, \Gamma_2 \fCenter \Delta_1, \Delta_2, A \nor B$
\DisplayProof\\[2ex]
& 
\Axiom$\Gamma_0 \fCenter \Delta_0, A \nor B$
\Axiom$\Gamma_1 \fCenter \Delta_2, A, B$
\RightLabel{$\Elim\nor$}
\BinaryInf$\Gamma_0, \Gamma_1 \fCenter \Delta_0, \Delta_1$
\DisplayProof \\[2ex]
A \xor B &
\Axiom$\Gamma_1 \fCenter \Delta_1, A, B$
\Axiom$A, B, \Gamma_2 \fCenter \Delta_2$
\RightLabel{$\Intro\xor$}
\BinaryInf$\Gamma_1, \Gamma_2 \fCenter \Delta_1, \Delta_2, A \xor B$
\DisplayProof \\[2ex]
& 
\Axiom$\Gamma_0 \fCenter \Delta_0, A \xor B$
\Axiom$A, \Gamma_1 \fCenter \Delta_1, B$
\Axiom$B, \Gamma_2 \fCenter \Delta_2, A$
\RightLabel{$\Elim\xor$}
\TrinaryInf$\Gamma_0, \Gamma_1, \Gamma_2 \fCenter \Delta_0, \Delta_1, \Delta_2$
\DisplayProof \\[2ex]
A \lif B / C & 
\Axiom$A, \Gamma_1 \fCenter \Delta_1, B$
\Axiom$\Gamma_2 \fCenter \Delta_2, A, C$ 
\RightLabel{$\Intro\ite$}
\BinaryInf$\Gamma_1,\Gamma_2 \fCenter \Delta_1, \Delta_2, A \lif B/C$
\DisplayProof\\[2ex]
&
\Axiom$\Gamma_0 \fCenter \Delta_0, A \lif B/C$
\Axiom$A, B, \Gamma_1 \fCenter \Delta_1$
\Axiom$C, \Gamma_2 \fCenter \Delta_2, A$
\RightLabel{$\Elim\ite$}
\TrinaryInf$\Gamma_0, \Gamma_1, \Gamma_2 \fCenter \Delta_0, \Delta_1, \Delta_2$
\DisplayProof\\[2ex]
\hline\hline
\end{array}
\]
\end{table}

\begin{lem}
In any derivation~$\delta$ in $\Nmsl X$, we can replace a label~$x$ uniformly
by another label~$y$ not already used in the derivation, and obtain a
correct derivation~$\delta[y/x]$.
\end{lem}

\begin{proof}
By induction on the height of the derivation.
\end{proof}

\begin{prop}
If $\delta$ is an $\Nmsl X$ derivation of $\Gamma \fCenter \Delta$,
there is a derivation in $\Nms X$ of $\Gamma' \fCenter \Delta$ where
$\Gamma'$ is $\Gamma$ with labels removed.
\end{prop}

\begin{proof}
By induction on the height of $\delta$. Simply remove labels, adding
$\Left\Contr$, $\Left\Exch$, and $\Right\Exch$ inferences where
required. For each inference in $\delta$ in which not all labelled
formulas $x \colon A$ are discharged, add a $\Left\Weak$ inference to
the corresponding premise before applying the rule.
\end{proof}

\begin{prop}
If $\delta$ is an $\Nms X$ derivation of $\Gamma \fCenter \Delta$, there
is a derivation $\delta'$ in $\Nmsl X$ of $\Gamma' \fCenter \Delta$
where if $x \colon A$ is a labelled occurrence of $A$ in $\Gamma'$,
$\Gamma$ contains~$A$.
\end{prop}

\begin{proof}
In a first pass, we assign sets of labels to the formulas in the
contexts in~$\delta$ inductively:

If $A \fCenter A$ is an initial sequent, replace it with $\{x\} \colon
A \fCenter A$ for some label $x$.

If $A$ is the weakened formula in a $\Left\Weak$ inference, assign
$\emptyset$ to it. 

If label sets $l_1$,and $l_2$, both $\neq \emptyset$, have been
assigned to the auxiliary formulas $A, A$ in a $\Left\Contr$
inference, replace every label set $l_1$ and $l_2$ in the derivation
ending in the premise by the set $l_1 \cup l_2$ and assign $l_1 \cup
l_2$ to $A$ in the conclusion.

In all other inferences, first uniformly replace labels throughout the
derivations ending in the premises to ensure that the label sets
appearing in any two derivations are disjoint.  Then assign the same
label sets to the formulas in the conclusion as the corresponding
formulas in the premises.

If $l$, $l'$ are two different label sets in the result, $l \cap l' =
\emptyset$. We may assume that the labels $x_i$ are linearly ordered;
let $\min(l)$ be the least label in~$l$ in this ordering.

We define the translation by induction on the height of the labelled
derivation~$\delta$.  The translation has the following property: if
the labelled $\delta$ ends in $\Gamma \fCenter \Delta$ and its
translation $\delta'$ ends in $\Gamma' \fCenter \Delta'$, then (a) if $l
\colon A \in \Gamma$ with $l \neq \emptyset$, then $\min(l) \colon A
\in \Gamma'$, and (b) if $x \colon A \in \Gamma'$ then for some $l$,
$l \colon A \in \Gamma$ with $x = \min(l)$, and (c) $\Gamma'$ contains
each $x \colon A$ at most once. In each step below, it is easily
verified that (a), (b), and (c) hold.

If $\delta$ is just an initial sequent $l \colon A \fCenter A$,
$\delta'$ is $\min(l) \colon A \fCenter A$. 

If $\delta$ ends in a $\Right\Weak$, $\Right\Contr$, or $\cut$ inference, apply $\Right\Weak$,
$\Right\Contr$, or $\cut$, respectively, to the translation of the premise.

If $\delta$ ends in $\Right\Exch$, or in a $\Left\Weak$ or
$\Left\Contr$ inference with principal or weakened formula $l \colon
A$, $\delta'$ is the translation of its premise (i.e., we remove left
weakenings and contractions). In the case of $\Left\Contr$, by
induction hypothesis, the antecedent of the translation of the premise
already contains $\min(l) \colon A$ at most once.

Suppose $\delta$ ends in a logical inference,
\[
\Axiom$\Gamma_0 \fCenter \Delta, \opa(\vec A)$
\AxiomC{$\Pi_1, \Gamma_1 \fCenter \Delta_1, \Lambda_1 
\quad  \dots \quad
  \Pi_n, \Gamma_n \fCenter \Delta_n, \Lambda_n$}
\RightLabel{$\Elim\opa$}
\BinaryInf$\Gamma_0, \Gamma_1, \dots, \Gamma_n \fCenter \Delta_0, \Delta_1, \dots, \Delta_n$
\DisplayProof
\]
Let $\delta_i'$ be the translations of the derivations ending in the
premises. Each $\delta'_i$ proves $\Pi_i', \Gamma_i' \fCenter \Delta_i,
\Lambda_i$ where $\Pi_i'$ is a set of labelled formulas.  For each $l
\colon A \in \Pi_i$, $\min(l) \colon A \in \Pi_i'$ iff $l \neq
\emptyset$, and if $x \colon A \in \Pi_i'$ then $x = \min(l)$.  If
different labelled formulas $l \colon A$ and $l' \colon A$ appear in
$\Pi_i$ and $\Pi_j$ but instantiate the same schematic formula of the
rule, then $\min(l) \colon A$ and $\min(l') \colon A$ appear in
$\Pi_i'$ and $\Pi_j'$.  Replace $\min(l') \colon A$ in the entire
derivation~$\delta_j'$ ending in the second premise by $\min(l) \colon
A$ everywhere it occurs. In this way we obtain $\delta_i''$ in which
labelled formulas $x \colon A$ which instantiate the same schematic
variable have the same label.  Adding $\Elim\opa$ to the
resulting $\delta_i''$ results in a correct inference, if we label the
inference with the list of labels appearing in the~$\Pi_i$.  Note that
when $\Elim\opa$ discharges a formula $A$ which is obtained by weakening,
it was originally labelled by $\emptyset$. It either still was
labelled with $\emptyset$ at the end, or its label set was combined
with that of another occurrence of $A$ with which it was later
contracted. In the first case it no longer appears in the premise of
the translated derivation, the corresponding inference thus vacuously
discharges~$A$.

The case of $\Intro\opa$ is treated the same way.
\end{proof}

\section{Specialized elimination rules}
\label{sec:specialize}

In order to get the familiar elimination rules for natural deduction,
the general elimination rules obtained from the $\Left\opa$ rule must be
simplified.  For instance, the general elimination rule for $\lif$
based on $\Left\lif$ is
\[
\Axiom$\Gamma \fCenter \Delta, A \lif B$
\Axiom$\Gamma \fCenter \Delta, A$
\Axiom$B, \Gamma \fCenter \Delta$
\RightLabel{$\Elim\lif$}
\TrinaryInf$\Gamma \fCenter \Delta$
\DisplayProof
\]
We can obtain a specialized rule by removing a premise which only
discharges a single assumption, and instead add the discharged assumption to
the conclusion.  In the case of $\lif$, for instance, we observe that
if $B \in \Delta$, the right minor premise $B, \Gamma \fCenter \Delta$
contains the initial sequent $B \fCenter B$ as a subsequent and is thus
always derivable.  The additional formula $B$ is then of course also
part of~$\Delta$ in the conclusion.
The simplified rule $\Elim\lif'$ is the familiar modus ponens rule:
\[
\Axiom$\Gamma \fCenter \Delta, A \lif B$ 
\Axiom$\Gamma \fCenter \Delta, A$
\RightLabel{$\Elim\lif'$}
\BinaryInf$\Gamma \fCenter \Delta, B$
\DisplayProof
\]

In general, a premise of $\Elim\opa$ of the form $A_i, \Gamma \fCenter
\Delta$ can be removed while $A_i$ is added to the right side of the
conclusion to obtain a simplified rule~$\Elim\opa'$.  For if $\Pi, \Gamma
\fCenter \Delta, \Lambda$ is derivable (the general form of any of the
premises), so is $\Pi, \Gamma \fCenter \Delta, A_i, \Lambda$ by $\Right\Weak$ and
$\Right\Exch$.  The sequent corresponding to the removed premise is $A_i,
\Gamma \fCenter \Delta, A_i$ which is derivable from the initial sequent
$A_i \fCenter A_i$ by weakenings and exchanges alone. E.g.,
\[
\Axiom$\Gamma \fCenter \Delta, \opa(A_1, \dots, A_n)$
\AxiomC{$\cdots \quad 
  \Pi_i, \Gamma \fCenter \Delta, \Lambda_i 
  \quad \cdots$}
\RightLabel{$\Elim\opa'$}
\BinaryInf$\Gamma \fCenter \Delta, A_i$
\DisplayProof
\] turns into \[
\Axiom$\Gamma \fCenter \Delta, \opa(A_1, \dots, A_n)$
\doubleLine
\UnaryInf$\Gamma \fCenter \Delta, A_i, \opa(A_1, \dots, A_n)$
\AxiomC{\dots}
\Axiom$\Pi_i, \Gamma \fCenter \Delta, \Lambda_i$ 
\doubleLine
\UnaryInf$\Pi_i, \Gamma \fCenter \Delta, A_i, \Lambda_i \dots$
\AxiomC{\dots}
\Axiom$A_i \fCenter A_i$
\doubleLine
\UnaryInf$A_i, \Gamma \fCenter \Delta, A_i$
\RightLabel{$\Elim\opa$}
\QuinaryInf$\Gamma \fCenter \Delta, A_i$
\DisplayProof
\]
Thus, any inference using the specialized rule $\Elim\opa'$ can be
replaced by the original rule~$\Elim\opa$.  Conversely, if the premises
of the original rule are derivable, the specialized rule applied to
the premises still present in the specialized rule proves $\Gamma
\fCenter \Delta, A_i$.  The removed premise of the original rule is
$A_i, \Gamma \fCenter \Delta$. Using $\cut$ and contraction, we obtain a
derivation of $\Gamma \fCenter \Delta$. E.g.,
\[
\Axiom$\Gamma \fCenter \Delta, \opa(A_1, \dots, A_n)$
\AxiomC{$\cdots \quad 
  \Pi_i, \Gamma \fCenter \Delta, \Lambda_i 
  \quad \cdots \quad 
  A_i, \Gamma \fCenter \Delta$}
\RightLabel{$\Elim\opa$}
\BinaryInf$\Gamma \fCenter \Delta$
\DisplayProof
\] 
turns into 
\[
\Axiom$\Gamma \fCenter \Delta, \opa(A_1, \dots, A_n)$
\AxiomC{$\cdots \quad 
      \Pi_i, \Gamma \fCenter \Delta, \Lambda_i
      \quad\cdots$}
\RightLabel{$\Elim\opa'$}
\BinaryInf$\Gamma \fCenter \Delta, A_i$
\Axiom$A_i, \Gamma \fCenter \Delta$
\RightLabel{$\cut$}
\BinaryInf$\Gamma, \Gamma \fCenter \Delta, \Delta$
\doubleLine
\UnaryInf$\Gamma \fCenter \Delta$
\DisplayProof
\]

Since we allow multiple formulas in the succedent, this generalizes to
multiple premises that only discharge single assumptions. For instance, we can
simplify~$\Elim\lor$
\[
\Axiom$\Gamma \fCenter \Delta, A \lor B$
\Axiom$A, \Gamma \fCenter \Delta$
\Axiom$B, \Gamma \fCenter \Delta$
\RightLabel{$\Elim\lor$}
\TrinaryInf$\Gamma \fCenter \Delta$
\DisplayProof
\]
to get
\[
\Axiom$\Gamma \fCenter \Delta, A \lor B$
\RightLabel{$\Elim\lor'$}
\UnaryInf$\Gamma \fCenter \Delta, A, B$
\DisplayProof
\]
Note that the simulation of the original general rule by the
specialized elimination rule requires a cut. It should thus not
surprise that general elimination rules are proof-theoretically better
behaved than the specialized rules.

It is also possible using the same idea to specialize rules by removing
premises which only contain a single formula in the succedent. In that
case, the corresponding formula must be added to the \emph{antecedent}
of the conclusion.  Here is an example. The general $\Elim\mid$ rule 
\[
\Axiom$\Gamma_0 \fCenter \Delta_0, A \mid B$
\Axiom$\Gamma_1 \fCenter \Delta_1, A$
\Axiom$\Gamma_2 \fCenter \Delta_2, B$
\RightLabel{$\Elim\mid$}
\TrinaryInf$\Gamma_0, \Gamma_1, \Gamma_2 \fCenter \Delta_0, \Delta_1, \Delta_2$
\DisplayProof
\]
specializes to the two rules
\[
\Axiom$\Gamma_0 \fCenter \Delta_0, A \mid B$
\Axiom$\Gamma_2 \fCenter \Delta_2, B$ 
\RightLabel{$\Elim\mid'_1$}
\BinaryInf$A, \Gamma_0, \Gamma_2 \fCenter \Delta_0, \Delta_2$
\DisplayProof
\qquad
\Axiom$\Gamma_0 \fCenter \Delta_0, A \mid B$
\Axiom$\Gamma_1 \fCenter \Delta_1, A$
\RightLabel{$\Elim\mid'_2$}
\BinaryInf$B, \Gamma_0, \Gamma_1 \fCenter \Delta_0, \Delta_1$
\DisplayProof
\]
and further to the single rule
\[
\Axiom$\Gamma_0 \fCenter \Delta_0, A \mid B$
\RightLabel{$\Elim\mid'$}
\UnaryInf$A, B, \Gamma_0 \fCenter \Delta_0$
\DisplayProof
\]

Sequent-style natural deduction is closely connected to standard
natural deduction, in which assumptions are not collected in the
antecedent of a sequent but are simply formulas at the top of a proof
tree, possibly marked as discharged. In the standard formalism, such a
specialized rule is difficult to accommodate, since it would amount to
a rule that allows one to \emph{add} undischarged assumptions to the
proof tree which don't already occur in the tree.

The restriction to a single formula in the antecedent of the
specialized premise is essential.  We \emph{cannot} specialize a rule
by removing a premise which discharges two assumptions by putting both
assumptions into the succedent of the conclusion.  The reason is that
from the conclusion $\Gamma \fCenter \Delta, A_i, A_j$ we cannot
recover $\Gamma \fCenter \Delta$ by cuts together with the removed
premise $A_i, A_j, \Gamma \fCenter \Delta$. We can, however, replace
the rule by a set of specialized rules, one for each assumption
formula discharged in the specialized premise.  This corresponds
exactly to first splitting the $\Left\opa$ rule, considering the
generalized $\Elim\opa$ rules based on these split rules, and then
specializing the resulting rules as before. Consider for example the
rules for $\land$. The standard $\Left\land$ rule has the single
premise $A, B, \Gamma \fCenter \Delta$.  It corresponds to the general
elimination rule
\[
\Axiom$\Gamma \fCenter \Delta, A \land B$ 
\Axiom$A, B, \Gamma \fCenter \Delta$
\RightLabel{$\Elim\land$}
\BinaryInf$\Gamma \fCenter \Delta$
\DisplayProof
\]
The rule $\Left\land$ can be split into two $\Left\land$ rules with premises
$A, \Gamma \fCenter \Delta$ and $B, \Gamma \fCenter \Delta$, respectively.
The corresponding general $\Elim\land$ rules are
\[
\Axiom$\Gamma \fCenter \Delta, A \land B$
\Axiom$A, \Gamma \fCenter \Delta$ 
\RightLabel{$\Elim\land_1$}
\BinaryInf$\Gamma \fCenter \Delta$
\DisplayProof
\quad\text{and}\quad
\Axiom$\Gamma \fCenter \Delta, A \land B$
\Axiom$B, \Gamma \fCenter \Delta$
\RightLabel{$\Elim\land_2$}
\BinaryInf$\Gamma \fCenter \Delta$
\DisplayProof\]
which can be specialized to the familiar two $\land$ elimination
rules:
\[
\Axiom$\Gamma \fCenter \Delta, A \land B$ 
\RightLabel{$\Elim\land_1'$}
\UnaryInf$\Gamma \fCenter \Delta, A$
\DisplayProof
\qquad\text{and}\qquad
\Axiom$\Gamma \fCenter \Delta, A \land B$
\RightLabel{$\Elim\land_2'$}
\UnaryInf$\Gamma \fCenter \Delta, B$
\DisplayProof
\]

\begin{table}
  \caption{Some specialized and split elimination rules for some unusual connectives}
\[
\begin{array}{cc}
\hline\hline
\text{connective} & \text{intro/elim rule} \\ \hline\\ 
A \nif B &  
\Axiom$\Gamma_0 \fCenter \Delta_0, A \nif B$
\Axiom$\Gamma_1 \fCenter \Delta_1, B$
\RightLabel{$\Elim\nif_1$}
\BinaryInf$\Gamma_0, \Gamma_1 \fCenter \Delta_0, \Delta_1$
\DisplayProof \\[2ex] 
& 
\Axiom$\Gamma_0 \fCenter \Delta_0, A \nif B$
\RightLabel{$\Elim\nif_2$}
\UnaryInf$\Gamma_0 \fCenter \Delta_0, A$
\DisplayProof \\[2ex] 
A \xor B &
\Axiom$\Gamma_0 \fCenter \Delta_0, A \xor B$
\RightLabel{$\Elim\xor_1$}
\UnaryInf$\Gamma_0 \fCenter \Delta_0, A, B$
\DisplayProof \\[2ex]
& 
\Axiom$\Gamma_0 \fCenter \Delta_0, A \xor B$
\Axiom$\Gamma_1 \fCenter \Delta_1, B$
\Axiom$\Gamma_2 \fCenter \Delta_2, A$
\RightLabel{$\Elim\xor_2$}
\TrinaryInf$\Gamma_0, \Gamma_1, \Gamma_2 \fCenter \Delta_0, \Delta_1, \Delta_2$
\DisplayProof \\[2ex]
A \lif B / C & 
\Axiom$\Gamma_0 \fCenter \Delta_0, A \lif B/C$
\Axiom$C, \Gamma_2 \fCenter \Delta_2, A$
\RightLabel{$\Elim\ite_1$}
\BinaryInf$\Gamma_0, \Gamma_2 \fCenter \Delta_0, \Delta_2, B$
\DisplayProof\\[2ex]
& 
\Axiom$\Gamma_0 \fCenter \Delta_0, A \lif B/C$
\Axiom$A, \Gamma_1 \fCenter \Delta_1$
\RightLabel{$\Elim\ite_2$}
\BinaryInf$\Gamma_0, \Gamma_1 \fCenter \Delta_0, \Delta_1, C$
\DisplayProof\\[2ex]
  \hline\hline
\end{array}
\]
\end{table}

It is possible to generate a general elimination rule from a
specialized one by reverse-engineering the specialization process.
Suppose we have a specialized rule $\Elim\opa'$:
\[
\Axiom$\Gamma \fCenter \Delta, \opa(\vec A)$
\AxiomC{$\cdots \quad
  \Pi_i, \Gamma \fCenter \Delta, \Lambda_i  
  \quad\cdots$}
\RightLabel{$\Elim\opa'$}
\BinaryInf$\Gamma \fCenter \Delta, A_i$
\DisplayProof
\]
Then the corresponding general elimination rule has the same premises
as $\Elim\opa'$ plus an additional premise $A_i, \Gamma \fCenter \Delta$,
and the conclusion leaves out the~$A_i$:
\[
\Axiom$\Gamma \fCenter \Delta, \opa(\vec A)$
\AxiomC{$\cdots\quad
  \Pi_i, \Gamma \fCenter \Delta, \Lambda_i 
  \quad\cdots$} 
\Axiom$A_i, \Gamma \fCenter \Delta$
\RightLabel{$\Elim\opa$}
\TrinaryInf$\Gamma \fCenter \Delta$
\DisplayProof
\]
Since $\Elim\opa'$ is the result of $\Elim\opa$ specialized in the premise
$A_i, \Gamma \fCenter \Delta$, the two rules are equivalent.

\section{Cut elimination and substitution}
\label{sec:cutel}

Gentzen's cut elimination method proceeds by permuting inference rules
with the cut rule until there is a topmost cut where the cut formula
is introduced in both premises of the cut by a right and a left
inference, respectively.  Such topmost cuts can then be reduced to
cuts with cut formulas of a lower degree.  We'll now show that the
fact that this is possible is no accident; in fact it holds whenever
we have a calculus obtained by our procedure.  The key fact here is
that the original clause sets $\C(\opa)^+$ and $\C(\opa)^-$ are not
jointly satisfiable.  As an unsatisfiable clause set it has a
\emph{resolution refutation}, which can be immediately translated into
a derivation of the empty sequent from the sequents corresponding to
the clauses in $\C(\opa)^+ \cup \C(\opa)^-$. 

The same is true for clause sets corresponding to split rules, since
each clause in them is a subset of a clause in $\C(\opa)^+ \cup
\C(\opa)^-$.  

\begin{thm}
The cut elimination theorem holds for $\L X$.
\end{thm}

\begin{proof}
For a detailed proof, see Theorem 4.1 of \cite{BaazFermullerZach1994}.
We give a sketch only. As in \cite{Gentzen1934}, we observe that the
\cut{} rule is equivalent to the \mix{} rule,
\[
\Axiom$\Gamma \fCenter \Delta$
\Axiom$\Theta \fCenter \Xi$
\RightLabel{$\mix: A$}
\BinaryInf$\Gamma, \Lambda^* \fCenter \Theta^*, \Xi$
\DisplayProof
\]
where $\Delta^*$ and $\Theta^*$ are $\Delta$ and $\Xi$, respectively,
with every occurrence of the mix formula~$A$ removed. We show that a
proof with a single \mix{} inference as its last inference can be
transformed into one without \mix. The result then follows by
induction on the number of applications of \mix{} in the proof.

We introduce two measures on proofs ending in a single \mix: The
degree is the degree of the \mix{} formula~$A$. The left rank is the
maximum number of consecutive sequents on a branch ending in the left
premise that contains the \mix{} formula~$A$ on the right, counting
from the left premise; and similarly for the right rank. The rank of
the \mix{} is the sum of the left and right rank. The proof then
proceeds by double induction on the rank and degree. We distinguish
cases according to the inferences ending in the premises of the \mix{}
inference and show how, in every case, either that the degree of the
\mix{} can be reduced, or else that there is a proof using a \mix{} of
the same degree but of lower rank.

The first case occurs when the \mix{} formula~$A$ appears only once in
$\Delta$ and~$\Theta$ and as the principal formula of a~$\Right\opa$
inference that ends in the left premise $\Gamma \fCenter \Delta$ and of
a~$\Left\opa$ inference that ends in the right premise $\Theta \fCenter
\Xi$. In all other cases, the \mix{} inference is switched with the
last inference on the left or right side, thus reducing the rank.

In the critical first case, the proof ends in
\[
\Axiom$\cdots\quad \Pi_i, \Gamma \fCenter \Delta, \Lambda_i \quad\cdots$
\RightLabel{$\Right\opa$}
\UnaryInf$\Gamma \fCenter \Delta, \opa(\vec A)$
\Axiom$\cdots\quad \Pi_i', \Theta \fCenter \Xi, \Lambda_i' \quad\cdots$
\RightLabel{$\Left\opa$}
\UnaryInf$\opa(\vec A), \Theta \fCenter \Xi$
\RightLabel{$\mix:\opa(\vec A)$}
\BinaryInf$\Gamma, \Theta \fCenter \Delta, \Xi$
\DisplayProof
\]
Remove the side formulas $\Gamma$, $\Delta$, $\Theta$, $\Xi$ from the
premises and consider the set of sequents $\Pi_i \fCenter \Lambda_i$
and $\Pi_i' \fCenter \Lambda_i'$. These are just the clauses
$\C(\opa)^+$ and $\C(\opa)^-$, which together form an unsatisfiable
set of Kowalski clauses. A resolution refutation of this clause set
yields a derivation of the empty sequent from these sequents using
only \mix{} inferences on the~$A_i$. This derivation in turn, by
adding the side formulas $\Gamma$, $\Delta$, $\Theta$, $\Xi$
appropriately, and adding some exchanges and contractions at the end,
yields a derivation of the sequent $\Gamma, \Theta \fCenter \Delta,
\Xi$. The remaining \mix{} inferences are all of lower degree, and so
the induction hypothesis applies.

In the other cases we show that the end-sequent has a cut-free derivation
by appealing to the second clause of the induction hypothesis, namely,
that proofs ending in \mix{} inferences of lower rank can be
transformed into \mix-free proofs. For instance, suppose the right
premise ends in $\Left\opa$ but its principal formula~$\opa(\vec A)$ is
not the only occurrence of the \mix{} formula in the antecedent of the
right premise:
\[
\Axiom$\Gamma \fCenter \Delta$
\Axiom$\cdots\quad \Pi_i', \Theta \fCenter \Xi, \Lambda_i' \quad\cdots$
\RightLabel{$\Left\opa$}
\UnaryInf$\opa(\vec A), \Theta \fCenter \Xi$
\RightLabel{$\mix: \opa(\vec A)$}
\BinaryInf$\Gamma, \Theta^* \fCenter \Delta^*, \Xi$
\DisplayProof
\]
We may assume that the \mix{} formula $\opa(\vec A)$ does not occur
in~$\Gamma$, since otherwise we can obtain a \mix-free proof of the
end-sequent from the right premise of the \mix{} using only weakening
and exchanges.

If we now apply \mix{} to the left premise and to each of the premises
of the $\Left\opa$ inference directly we have:
\[
\Axiom$\Gamma \fCenter \Delta$
\Axiom$\Pi_i', \Theta \fCenter \Xi, \Lambda_i'$
\RightLabel{$\mix: \opa(\vec A)$}
\BinaryInf$\Gamma, \Pi_i', \Theta^* \fCenter
    \Delta^*, \Xi, \Lambda_i'$
\DisplayProof
\]
Since the subproof leading to the premise on the right no longer
contains the $\Left\opa$ inference, the right rank is reduced by~$1$, and
so the induction hypothesis applies. Each of these derivations can be
transformed into a \mix-free derivation. We can apply $\Left\opa$ and
\mix{} again:
\[
\Axiom$\Gamma \fCenter \Delta$
\AxiomC{$\cdots$}
\Axiom$\Gamma, \Pi_i', \Theta^* \fCenter \Delta^*, \Xi, \Lambda_i'$
\doubleLine
\RightLabel{$\Left\Exch$}
\UnaryInf$\Pi_i', \Gamma, \Theta^* \fCenter \Delta^*, \Xi, \Lambda_i'$
\AxiomC{$\cdots$}
\RightLabel{$\Left\opa$}
\TrinaryInf$\opa(\vec A), \Gamma, \Theta^* \fCenter \Delta^*, \Xi$
\RightLabel{$\mix: \opa(\vec A)$}
\BinaryInf$\Gamma, \Gamma, \Theta^* \fCenter \Delta^*, \Delta^*, \Xi$
\DisplayProof
\]
The right rank of this \mix{} is~$1$ since $\opa(\vec A)$ does not
occur in $\Pi_i', \Gamma, \Theta^*$, so the rank is lower than that of
the original \mix{} and the induction hypothesis again applies:
$\Gamma, \Gamma, \Theta^* \fCenter \Delta^*, \Delta^*, \Xi$ has a
proof without \mix. We can obtain the original end-sequent $\Gamma,
\Theta^* \fCenter \Delta^*, \Xi$ from $\Gamma, \Gamma, \Theta^*
\fCenter \Delta^*, \Delta^*, \Xi$ by exchanges and contractions.
\end{proof}

To illustrate the use of resolution refutations in the first case,
consider a proof that ends in
\[
\Axiom$\Gamma \fCenter \Delta, A$
\Axiom$\Gamma \fCenter \Delta, B$
\RightLabel{$\Right\land$}
\BinaryInf$\Gamma \fCenter \Delta, A \land B$
\Axiom$A, B, \Theta \fCenter \Xi$
\RightLabel{$\Left\land$}
\UnaryInf$A \land B, \Theta \fCenter \Xi$
\RightLabel{$\mix: A \land B$}
\BinaryInf$\Gamma, \Theta \fCenter \Delta, \Xi$
\DisplayProof
\]
The corresponding set of Kowalski clauses is
\[
\fCenter A, \qquad \fCenter B, \qquad A, B \fCenter
\]
(i.e., $\{A\}$, $\{B\}$, $\{\lnot A, \lnot B\}$). A resolution
refutation is
\[
\Axiom$\fCenter B$
\Axiom$\fCenter A$
\Axiom$A, B \fCenter$
\BinaryInf$B \fCenter$
\BinaryInf$\fCenter$
\DisplayProof
\]
to which we add the side formulas  appropriately, and exchanges and
contraction inferences at the end, to obtain:
\[
\Axiom$\Gamma \fCenter \Delta, B$
\Axiom$\Gamma \fCenter \Delta, A$
\Axiom$A, B, \Theta \fCenter \Xi$
\RightLabel{$\mix$}
\BinaryInf$\Gamma, B, \Theta \fCenter \Delta, \Xi$
\RightLabel{$\mix$}
\BinaryInf$\Gamma, \Gamma, \Theta \fCenter \Delta, \Delta, \Xi$
\doubleLine
\UnaryInf$\Gamma, \Theta \fCenter \Delta, \Xi$
\DisplayProof
\]

In the natural deduction system $\Nms X$, the cut rule can also be
eliminated.  In this case, however, the proof is much simpler than 
for $\L X$. In fact, in $\Nms X$ the cut rule simply
corresponds to substituting derivations for assumptions.  Since $\Nms 
X$ has no rules which introduce a formula in the antecedent of a
sequent in a derivation, every formula in the antecedent is either
weakened, or it stems from an initial sequent.  Whenever it is
weakened, we can derive the conclusion of the cut inference from the
premise of the weakening inference already.  If it stems from an
initial sequent, we can replace the initial sequent by the derivation
ending in the left premise, and adding the context formulas $\Gamma$
and $\Delta$ to the antecedent and succedent of every inference below
it, adjusting any inferences with weakenings as necessary.

Suppose $\delta$ is a derivation of $\Gamma \fCenter \Delta, A$ and
$\delta'$ is a derivation of $A, \Pi \fCenter \Lambda$, and suppose that
neither $\delta$ nor $\delta'$ contains a cut. We can define the
\emph{substitution}~$\delta'[\delta / A]$ with end sequent $\Pi, \Gamma
\fCenter \Delta, \Lambda$ recursively as follows.  We distinguish cases
according to the last inference of $\delta'$:

\begin{enumerate}
\item $\delta'$ is an initial sequent $A \fCenter A$. Then $\Pi =
  \emptyset$, $\Lambda = A$, and $\delta'[\delta/A]$ is~$\delta$.
\item $\delta'$ is an initial sequent $B \fCenter B$, but $B \neq
  A$. Then $\delta'[\delta/A]$ is
\[
\Axiom$B \fCenter B$
\doubleLine
\UnaryInf$B, \Gamma \fCenter \Delta, B$
\DisplayProof
\]
\item $\delta'$ ends in a $\Left\Weak$. The weakening formula is $A$, and the
  deduction ending in the premise is $\delta'' \colon \Pi \fCenter
  \Lambda$. Let $\delta'[\delta/A]$ be
\[
\AxiomC{}
\RightLabel{$\delta''$}
\Deduce$\Pi \fCenter \Lambda$
\doubleLine
\UnaryInf$\Pi, \Gamma \fCenter \Delta, \Lambda$
\DisplayProof
\]
\item $\delta'$ ends in any other rule: Then $\delta'[\delta/A]$ is
  obtained by substituting $\delta$ for $A$ in each of the premises,
  and then applying the same rule.  It is straightforward to verify
  that the result is a correct derivation of $\Pi, \Gamma \fCenter
  \Delta, \Lambda$.  Note in particular that in none of the remaining
  rules of $\Nms X$ can $A$ be a principal formula in the
  antecedent.
\end{enumerate}

\begin{prop}
The cut rule can be eliminated from derivations in $\Nms X$.
\end{prop}

\begin{proof}
By induction on the number of cut inferences in a
derivation~$\delta$. If there are no cuts in $\delta$, there is
nothing to prove. Suppose $\delta$ contains a cut inference; pick an
uppermost one.  Replace the subderivation
\begin{endproofeqnarray*}
%\centerAlignProof
\AxiomC{}
\RightLabel{$\delta'$}
\Deduce$\Gamma \fCenter \Delta, A$
\AxiomC{}
\RightLabel{$\delta''$}
\Deduce$A, \Pi \fCenter \Lambda$
\RightLabel{$\cut$}
\BinaryInf$\Gamma, \Pi \fCenter \Delta, \Lambda$
\DisplayProof
&\qquad\text{by}\qquad&
%\centerAlignProof
\AxiomC{}
\RightLabel{$\delta''[\delta'/A]$}
\Deduce$\Pi, \Gamma \fCenter \Delta, \Lambda$
\doubleLine
\UnaryInf$\Gamma, \Pi \fCenter \Delta, \Lambda$
\DisplayProof
\end{endproofeqnarray*}
\end{proof}

Substitutions can similarly be defined for $\Nmsl X$, the variant of
$\Nms X$ without left weakening and contraction in which antecedent
formulas are labelled. Suppose $\delta$ and $\delta'$ are $\Nmsl X$
derivations without labels in common, without $\cut$, and $\delta$ ends
in $\Gamma \fCenter \Delta, A$ and $\delta'$ ends in $x \colon A, \Pi
\fCenter \Lambda$.  We define a substituted translation
$\delta'[\delta/x \colon A]$ of $\Gamma', \Pi \fCenter \Delta, \Lambda$
with $\Gamma' \subseteq \Gamma$.

\begin{enumerate}
\item $\delta'$ is an initial sequent $x \colon A \fCenter A$. Then
  $\Lambda = \{A\}$. Let $\delta'[\delta/x \colon A]$ be $\delta$ plus
  $\Right\Weak$ to obtain a derivation of $\Gamma \fCenter \Delta, A$.
\item $\delta'$ is an initial sequent $y \colon B \fCenter B$, but $B
  \neq A$ or $x \neq y$. Then $\delta'[\delta/A]$ is $\delta'$ plus
  $\Right\Weak$ to obtain a derivation of $\Pi \fCenter \Delta, B$.
\item $\delta'$ ends in any other rule: Then $\delta'[\delta/x \colon
  A]$ is obtained by substituting $\delta$ for $x \colon A$ in each of
  the premises, and then applying the same rule. If the rule
  discharged an assumption $x \colon A$ no longer present in the
  antecedent of the respective substituted premise, in the resulting
  inference the corresponding discharge is vacuous and the label $y$
  removed from the list of discharged labels.
\end{enumerate}

It is straightforward to verify that the result is a correct
derivation of $\Gamma', \Pi \fCenter \Delta, \Lambda$.

\begin{prop}
The cut rule can be eliminated from derivations in $\Nmsl X$.
\end{prop}

\begin{proof}
By renaming the labels, we can guarantee that the labels in the
derivation of the premises are disjoint.
\end{proof}

\section{Normalization for multi-conclusion natural deduction}
\label{sec:normalization}

The cut elimination theorem in $\Nms X$ and $\Nmsl X$ is, as we have
seen, simply the result that derivations of a conclusion involving a
formula~$A$ may be substituted for assumptions of that formula.  In
the context of natural deduction, the more interesting result is that
every derivation reduces, via a sequence of local reduction steps, to
a normal form.  A derivation is in normal form if it involves no
``detours,'' i.e., no introductions of a formula by $\Intro\opa$ followed by
eliminations of the same formula by $\Elim\opa$.

We prove the normalization result for $\Nmsl X$ without \cut. Since
conclusions are now multi-sets of formulas, the situation vis-\`a-vis
detours in derivations is more complicated than in the
single-conclusion case. Specifically, in a single-conclusion
derivation any sequence of consecutive sequents beginning with an
$\Intro\opa$ and ending in an $\Elim\opa$ inference is a ``detour,''
i.e., a maximum segment. No two such segments can overlap, and the
beginning and end of any segment is uniquely determined. In the
multi-conclusion case, segments can overlap, i.e., it is possible to
have sequences of sequents in which a formula $\opa(\vec A)$ is
introduced, then another $\opb(\vec B)$ is introduced, then $\opa(\vec
A)$ is eliminated, and finally $\opb(\vec B)$ is eliminated. The
presence of the $\Right\Contr$ rule adds further complications: a
single $\Elim\opa$ inference can count as the end-point of more than
one maximum segment.\footnote{The proof of \cite{Cellucci1992} does
not consider this complication.} In order to deal with these
complications, we must track the specific formula occurrences
introduced and later eliminated in the definition of maximal segments.

\begin{defn}
A \emph{maximal segment} in a cut-free derivation~$\delta$ is a sequence 
of sequents $S_1$, \dots, $S_k$ and of occurrences $B_1$, \dots, $B_k$
of $\opa(\vec A)$ in the succedents of $S_1$, \dots, $S_k$,
respectively, with the following properties.
\begin{enumerate}
\item $S_1$ is the conclusion of a $\Intro\opa$ or $\Right\Weak$ inference with principal
  formula $B_1 = \opa(\vec A)$.
\item $S_i$ is a premise of an inference, $S_{i+1}$
  its conclusion, and $B_{i+1}$ is the occurrence of $\opa(\vec A)$
  corresponding to $B_i$ in $S_i$. Specifically, $B_{i+1}$ is a side
  formula if $B_i$ is, or the principal formula of $\Right\Contr$ if $B_i$ is
  one of the formula occurrences being contracted.
\item $S_k$ is the major premise of a $\Elim\opa$ inference with
  principal formula occurrence $B_k = \opa(A_1, \dots, A_n)$.
\end{enumerate}
A derivation is \emph{normal} if it contains no maximal segments.
\end{defn}

The formula $\opa(\vec A)$ is the \emph{maximal formula of the
segment}. The \emph{degree of the segment} is the degree of $\opa(\vec
A)$, i.e., the number of logical operators in~$\opa(\vec A)$. The
\emph{length} of the segment is $k$.

Note that although the principal formula of a $\Elim\opa$ inference
may be the last formula occurrence of more than one maximal segment
(if it passes through a $\Right\Contr$ inference), any principal
formula occurrence of $\Intro\opa$ is the first formula of at most one
maximal segment. Since we are dealing with multisets of formulas,
which formula occurrence in the conclusion of a rule corresponds to
which formula occurrence in which premise is underdetermined. We could
make this precise by introducing labels or moving to a calculus of
sequences in which it is determined; for simplicity we assume that we
have picked a way to associate corresponding formula occurrences in
the derivation we start from and keep this association the same
throughout the transformation.

\begin{thm}\label{thm:normalization}
Any cut-free derivation in $\Nmsl X$ normalizes, i.e., there is a
sequence of local proof transformations which ends in a normal
derivation.
\end{thm}

\begin{proof}
Let $m(\delta)$ be the maximal degree of segments in~$\delta$, and let
$i(\delta)$ be the number of segments of maximal degree $m(\delta)$.  We
proceed by induction on $(m(\delta), i(\delta))$.  

We first ensure that at least one suitable segment of maximal degree
is of length~$1$. There must be a segment $S_1$, \dots, $S_k$ with
maximal formula $\opa(A_1, \dots, A_n)$ of degree~$m(\delta)$ with the
following properties. (a) No premise of the $\Intro\opa$ or
$\Right\Weak$ rule with $S_1$ as conclusion lies on or below a segment
of degree $m(\delta)$. (b) No minor premise of the $\Elim\opa$ rule of
which $S_k$ is the major premise lies on or below a segment of degree
$m(\delta)$. If the length $k=1$ we are done. Otherwise, we decrease
the length of this maximal segment.

We consider cases according to the rule of which $S_k$ is the
conclusion.  Since the segment is of length $>1$, the formula
occurrence $B_k = \opa(\vec A)$ is not the principal formula of an
$\Intro\opa$ or $\Right\Weak$ inference. In each remaining case, this last
inference and the following $\Elim\opa$ inference can be permuted and the
length of the segment considered reduced.
\begin{enumerate}
\item The rule is $\Right\Weak$, but $B_k$ is not the weakened formula. Then
replace the inference
\[
\Axiom$\Gamma_0 \fCenter \Delta_0, \opa(\vec A)$
\RightLabel{$\Right\Weak$}
\UnaryInf$\Gamma_0 \fCenter \Delta_0, A, \opa(\vec A)$
\AxiomC{\dots}
\RightLabel{$\Elim\opa$}
\BinaryInf$\dots\Gamma_i\dots \fCenter \dots\Delta_i \dots, A$
\DisplayProof
\qquad\text{by}\qquad
\Axiom$\Gamma_0 \fCenter \Delta_0, \opa(\vec A)$
\AxiomC{\dots}
\RightLabel{$\Elim\opa$}
\BinaryInf$\dots\Gamma_i\dots \fCenter \dots\Delta_i\dots$
\RightLabel{$\Right\Weak$}
\UnaryInf$\dots\Gamma_i\dots \fCenter \dots\Delta_i\dots, A$
\DisplayProof\]
The length of any segment which ends with the $\Elim\opa$ inference, and
any segment which begins with the $\Right\Weak$ inference have decreased, and
others are unchanged. No new maximal segments are added.
\item The rule is $\Right\Contr$ with the principal formula not $B_k = \opa(\vec A)$.
Then replace the inferences
\[
\Axiom$\Gamma_0 \fCenter \Delta_0, A, A, \opa(\vec A)$ 
\RightLabel{$\Right\Contr$}
\UnaryInf$\Gamma_0 \fCenter \Delta_0, A, \opa(\vec A)$
\AxiomC{\dots} 
\RightLabel{$\Elim\opa$}
\BinaryInf$\dots\Gamma_i\dots \fCenter \dots\Delta_i\dots, A$
\DisplayProof
\qquad\text{by}\qquad
\Axiom$\Gamma_0 \fCenter \Delta_0, A, A, \opa(\vec A)$
\AxiomC{\dots}
\RightLabel{$\Elim\opa$}
\BinaryInf$\dots\Gamma_i\dots \fCenter \dots\Delta_i\dots, A, A$
\RightLabel{$\Right\Contr$}
\UnaryInf$\dots\Gamma_i\dots \fCenter \dots\Delta_i\dots, A$
\DisplayProof
\]
\item The rule is $\Right\Contr$ and the principal formula is $B_k = \opa(\vec A)$.
The segment ends in:
\[
\Axiom$\Gamma_0 \fCenter \Delta_0, \overbrace{\opa(\vec A)}^{B'_{k'-1}}, \overbrace{\opa(\vec A)}^{B_{k-1}}$ 
\RightLabel{$\Right\Contr$}
\UnaryInf$\Gamma_0 \fCenter \Delta_0, B'_{k'} = B_k= \opa(\vec A)$
\Axiom$\cdots\quad\Pi_i, \Gamma_i \fCenter \Delta_i, \Lambda_i \quad\cdots$
\RightLabel{$\Elim\opa$}
\BinaryInf$\dots\Gamma_i\dots \fCenter \dots\Delta_i\dots$
\DisplayProof
\]
where only one of the occurrences of $\opa(\vec A)$ in the premise of
$\Right\Contr$ is the maximal formula occurrence $B_{k-1}$ of the
segment. (The other occurrence of $\opa(\vec A)$ possibly belongs to
another maximal segment ending in the $\Elim\opa$ inference, say, it
is $B'_{k'-1}$ in some segment consisting of $B'_1$, \dots,
$B'_{k'}$.) Replace these inferences by  
\[
\Axiom$\Gamma_0 \fCenter \Delta_0, \overbrace{\opa(\vec A)}^{B'_{k'-1}}, \overbrace{\opa(\vec A)}^{B_{k-1}}$
\Axiom$\cdots\  \Pi_i, \Gamma_i \fCenter \Delta_i, \Lambda_i\ \cdots$
\RightLabel{$\Elim\opa$}
\BinaryInf$\dots\Gamma_i\dots \fCenter \dots\Delta_i,\dots, B'_{k'} = \opa(\vec A)$
\Axiom$\cdots\ \Pi_i, \Gamma_i \fCenter \Delta_i, \Lambda_i\ \cdots$
\RightLabel{$\Elim\opa$}
\BinaryInf$\dots\Gamma_i\dots \fCenter \dots\Delta_i,\Delta_i\dots$
\doubleLine
\UnaryInf$\dots\Gamma_i\dots \fCenter \dots\Delta_i\dots$
\DisplayProof
\]
This changes maximal segments as follows: The segment under
consideration now ends at $B_{k-1}$, and has therefore decreased in
length. Any segment ending in the original $\Elim\opa$ inference which
contained the other contracted formula occurrence $B'_{k'-1}$ now ends
at the lower $\Elim\opa$ inference, and is of the same length ($k'$) as
before. Because of property (b) of the topmost segment under
consideration, there are no segments of maximal degree passing through
or lying above the minor premises of the $\Elim\opa$ rule, and so the
duplication of the subproofs ending in these minor premises has no
effect on the number of segments of maximal degree.

\item The rule is a $\Intro\opb$ rule, with a principal formula $\opb(\vec
C)$, but the principal formula $\opb(\vec C)$ is not $B_{k-1} =
\opa(\vec A)$.  Exactly one of the premises must belong to the segment
being considered, without loss of generality assume the first is. Then
that premise is of the form $\Pi'_1, \Gamma'_1 \fCenter \Delta_1',
B_{k-1} = \opa(\vec A), \Lambda_1$.  Suppose the other premises of the
$\Intro\opb$ inference are the sequents $\Pi_j', \Gamma_j' \fCenter \Delta_j',
\Lambda_j'$ (the formulas in $\Pi_i'$ and $\Lambda_i'$ are the
auxiliary formulas). The conclusion of the inference is $\Gamma_1',
\dots, \Gamma_n' \fCenter \Delta_1', \dots, \Delta_n', B_k = \opa(\vec
A), \opb(\vec C)$. 

Let $\Pi_i, \Gamma_i \fCenter \Delta_i,\Lambda_i$ be the minor premises
of the $\Elim\opa$ rule.  If we let 
\begin{align*}
  \Gamma & = \Gamma_1, \dots, \Gamma_n \\
  \Gamma' & = \Gamma_1', \dots, \Gamma_m'\\ 
  \Delta & = \Delta_1, \dots, \Delta_n \\
  \Delta' & = \Delta_1', \dots, \Delta_m'
\end{align*}
the last inference in the segment has the following form:
\[
  \Axiom$\Pi_1', \Gamma_1' \fCenter \Delta_1', \overbrace{\opa(\vec A)}^{B_{k-1}}, \Lambda_1'$
  \Axiom$\cdots\ 
    \Pi_j', \Gamma_j' \fCenter \Delta_j', \Lambda_j' 
    \ \cdots$
  \RightLabel{$\Intro\opb$}
  \BinaryInf$\Gamma' \fCenter \Delta', B_k = \opa(\vec A), \opb(\vec C)$
  \Axiom$\cdots\  
  \Pi_i, \Gamma_i \fCenter \Delta_i, \Lambda_i
  \ \cdots$
\RightLabel{$\Elim\opa$}
\BinaryInf$\Gamma, \Gamma' \fCenter \Delta, \Delta', \opb(\vec C)$
\DisplayProof
\]
where $\dots \Pi_j', \Gamma_j' \fCenter \Delta_j', \Lambda_j' \dots$ are
the premises of the $\Intro\opb$ inference other than the first.  Replace
the inferences with
\[
  \Axiom$\Pi_1', \Gamma_1' \fCenter \Delta_1', \overbrace{\opa(\vec A)}^{B_{k-1}}, \Lambda_1'$ 
  \Axiom$\cdots\  
    \Pi_i, \Gamma_i \fCenter \Delta_i, \Lambda_i \ \cdots$ 
  \RightLabel{$\Elim\opa$}
  \BinaryInf$\Pi_1', \Gamma_1', \Gamma \fCenter \Delta_1', \Delta, \Lambda_1'$
  \Axiom$\cdots\
  \Pi_j', \Gamma_j' \fCenter \Delta_j', \Lambda_j' 
  \ \cdots$
\RightLabel{$\Intro\opb$}
\BinaryInf$\Gamma, \Gamma' \fCenter \Delta, \Delta', \opb(\vec C)$
\DisplayProof
\]
\item The last inference is $\opb E$. This is treated as the previous
  one, except that we now have to distinguish cases according to
  whether the segment runs through the major premise or one of the
  minor premises. Again, since $\opa(\vec A)$ must occur in the
  context of one of the premises, the $\Elim\opa$ rule can first be
  applied to that premise, and the $\opb E$ rule then to its original
  premises, with the one premise belonging to the segment replaced
  with the conclusion of the $\Elim\opa$ rule.
\end{enumerate}

Now consider a topmost maximal segment of length~$1$. We have that
$S_1 = S_k$ is both the conclusion of a $\Intro\opa$ or $\Right\Weak$
rule and the major premise of a $\Elim\opa$ rule with $B_1=B_k$ the
principal formula. In the second case, the segment in question has the
form
\[
  \Axiom$\Gamma_0 \fCenter \Delta_0$
  \RightLabel{$\Right\Weak$}
  \UnaryInf$\Gamma_0 \fCenter \Delta_0, \opa(\vec A)$
  \Axiom$\dots\quad 
  \Pi_i, \Gamma_i \fCenter \Delta_i, \Lambda_i 
  \quad\dots$
\RightLabel{$\Elim\opa$}
\BinaryInf$\Gamma_0, \dots \Gamma_i \dots \fCenter \Delta_0, \dots \Delta_i\dots$
\DisplayProof
\]
This segment can be replaced by the premise of the $\Right\Weak$ rule
followed by $\Right\Weak$ to add the $\Delta_i$ on the right:  
\[
  \Axiom$\Gamma_0 \fCenter \Delta_0$
  \RightLabel{$\Right\Weak$}
  \doubleLine
  \UnaryInf$\Gamma_0 \fCenter \Delta_0, \dots\Delta_i\dots$
  \DisplayProof
\]
Any inferences below which discharge assumptions in $\Gamma_1$, \dots,
$\Gamma_n$ are now vacuous.

The first case is the crucial one.  The segment is of the form
\[
  \Axiom$\dots\quad 
    \Pi_i, \Gamma_i \fCenter \Delta_i, \Lambda_i 
    \quad\dots$ 
  \RightLabel{$\Intro\opa$}
  \UnaryInf$\dots\Gamma_i\ldots \fCenter 
    \dots\Delta_i\dots, \opa(\vec A)$
  \Axiom$\dots\quad
  \Pi_j', \Gamma_j' \fCenter \Delta_j', \Lambda_j'
  \quad\dots$
\RightLabel{$\Elim\opa$}
\BinaryInf$\dots\Gamma_i\dots\Gamma_j'\ldots \fCenter
  \dots\Delta_i\dots \Delta_j'\dots$
\DisplayProof
\]
As in the cut elimination theorem, the premises $\Pi_i \fCenter
\Lambda_i$ of the $\Intro\opa$ rule together with the minor premises
$\Pi_i ' \fCenter \Lambda_i'$ of the $\Elim\opa$ rule (with the
context formulas $\Gamma_i$, $\Gamma_i'$, $\Delta_i$, $\Delta_i'$
removed) form an unsatisfiable set of Kowalski clauses. A resolution
refutation of these clauses results in a derivation of the empty
sequent $\fCenter$. If any literal is removed from such a derivation,
the resolution refutation can be pruned to yield a possible shorter
refutation.  For any assumption $x \colon A$ not discharged by the
$\Elim\opa$ and $\Intro\opa$ rules, remove the corresponding negative
literal from the initial set of clauses and prune the refutation.
(These cases correspond to what are usually called simplification
conversions.)

Now add the context formulas $\Gamma_i$, $\Gamma_j'$, $\Delta_i$,
$\Delta_j'$, to get a derivation of $\dots\Gamma_i,\Gamma_i'\dots
\fCenter \dots\Delta_i,\Delta_j\dots$ from the premises of the
$\Intro\opa$ and the minor premises of the $\Elim\opa$ rule using cuts
and structural rules only. We replace the maximal segment with this
derivation and eliminate the \cut s.

Because the maximal segment beginning with $\Intro\opa$ is topmost, and
the derivations ending in the minor premises of $\Elim\opa$ contain no
segments of maximal degree, the resulting derivation does not contain
any new segments of maximal degree. We have removed one segment of
maximal degree. If the segment in question was the only maximal
segment of degree~$m(\delta)$, the maximal degree of segments in the
resulting derivation is $< m(\delta)$. Otherwise, we have removed at
least one maximal segment, and thus reduced the number of the segments
of maximal degree by~$1$. Thus, either the maximal degree of the
resulting derivation is $< m(\delta)$ (if the segment was the only
segment of degree $m(\delta)$), or the number of segments of
maximal degree in the resulting derivation is $< i(\delta)$.
\end{proof}

To illustrate the differences to the case of normalization for
single-conclusion systems, consider the derivation fragment
\[
        \Axiom$\fCenter A, B, C$
        \RightLabel{$\Intro\lor$}
        \UnaryInf$\fCenter A \lor B^a, C$
        \Axiom$\fCenter D$
        \RightLabel{$\Right\Weak$}
        \UnaryInf$\fCenter A \lor B^b, D$
      \RightLabel{$\Intro\land$}
      \BinaryInf$\fCenter A \lor B^a, A \lor B^b, C \land D^c$
    \RightLabel{$\Right\Contr$}
    \UnaryInf$\fCenter A \lor B^{a,b}, C \land D^c$
    \Axiom$C, D \fCenter E$
    \RightLabel{$\Elim\land$}
    \BinaryInf$\fCenter A \lor B^{a,b}, E$
  \Axiom$A \fCenter F$
  \Axiom$B \fCenter F$
  \RightLabel{$\Elim\lor$}
  \TrinaryInf$\fCenter E, F$
\DisplayProof\]
This derivation fragment contains three overlapping segments, the
formulas labelled $a$, $b$, and $c$ respectively. Assuming that $A
\lor B$ and $C \land D$ are of the same degree, all three are segments
of maximal degree. The segment labelled $c$ with maximum formula $C
\land D$ is not topmost, since condition (a) is violated. We must pick
one of the other segments, say the one labelled $a$. (We assume that
no segments of maximal degree run through the minor premises of
$\Elim\lor$ and $\Elim\land$.) We first permute the $\Elim\lor$ rule
across the $\Elim\land$ rule to obtain:
\[
        \Axiom$\fCenter A, B, C$
        \RightLabel{$\Intro\lor$}
        \UnaryInf$\fCenter A \lor B^a, C$
        \Axiom$\fCenter D$
        \RightLabel{$\Right\Weak$}
        \UnaryInf$\fCenter A \lor B^b, D$
      \RightLabel{$\Intro\land$}
      \BinaryInf$\fCenter A \lor B^a, A \lor B^b, C \land D^c$
    \RightLabel{$\Right\Contr$}
    \UnaryInf$\fCenter A \lor B^{a,b}, C \land D^c$
    \Axiom$A \fCenter F$
    \Axiom$B \fCenter F$
  \RightLabel{$\Elim\lor$}
  \TrinaryInf$\fCenter C \land D^c, F$
  \Axiom$C, D \fCenter E$
\RightLabel{$\Elim\land$}
\BinaryInf$\fCenter E, F$
\DisplayProof
\]
Now the $\Elim\lor$ follows a contraction; we replace it with two $\lor
E$ rules, the first of which ends the segment $a$.
\[
          \Axiom$\fCenter A, B, C$
          \RightLabel{$\Intro\lor$}
          \UnaryInf$\fCenter A \lor B^a, C$
          \Axiom$\fCenter D$
          \RightLabel{$\Right\Weak$}
          \UnaryInf$\fCenter A \lor B^b, D$
        \RightLabel{$\Intro\land$}
        \BinaryInf$\fCenter A \lor B^a, A \lor B^b, C \land D^c$
        \Axiom$A \fCenter F$
        \Axiom$B \fCenter F$
      \RightLabel{$\Elim\lor$}
      \TrinaryInf$\fCenter A \lor B^b, C \land D^c, F$
      \Axiom$A \fCenter F$
      \Axiom$B \fCenter F$
    \RightLabel{$\Elim\lor$}
    \TrinaryInf$\fCenter C \land D^c, F, F$
    \RightLabel{$\Right\Contr$}
    \UnaryInf$\fCenter C \land D^c, F$
  \Axiom$C, D \fCenter E$
\RightLabel{$\Elim\land$}
\insertBetweenHyps{\hspace{-8em}}
\BinaryInf$\fCenter E, F$
\DisplayProof
\]
Note that only one premise of the $\Intro\lor$ rule contains the maximal
formula of the segment labelled $a$---the left one. We obtain:
\[
          \Axiom$\fCenter A, B, C$
          \RightLabel{$\Intro\lor$}
          \UnaryInf$\fCenter A \lor B^a, C$
          \Axiom$A \fCenter F$
          \Axiom$B \fCenter F$
          \RightLabel{$\Elim\lor$}
          \TrinaryInf$\fCenter F, C$
        \Axiom$\fCenter D$
        \RightLabel{$\Right\Weak$}
        \UnaryInf$\fCenter A \lor B^b, D$
      \RightLabel{$\Intro\land$}
      \BinaryInf$\fCenter A \lor B^b, F, C \land D^c$    
      \Axiom$A \fCenter F$
      \Axiom$B \fCenter F$
    \RightLabel{$\Elim\lor$}
    \TrinaryInf$\fCenter C \land D^c, F, F$
  \RightLabel{$\Right\Contr$}
  \UnaryInf$\fCenter C \land D^c, F$
  \Axiom$C, D \fCenter E$
\RightLabel{$\Elim\land$}
\insertBetweenHyps{\hspace{-8em}}
\BinaryInf$\fCenter E, F$
\DisplayProof
\]
The segment labelled $a$ is now of length $1$ and can be removed. The
segment labelled $b$ is the new topmost segment of maximal degree; the
segment labelled $c$ is considered when that one is removed.

The reader familiar with the normalization proof of \cite{Prawitz1965}
will of course realize that the structure of the preceding proof
mirrors that of Prawitz's proof very closely.  By generalizing the
proof, however, we see that its success does not at all depend on the
specific logical inference rules used.  The crucial steps are
eliminating segments of length~1, and permuting elimination rules
upward across any inference in which the formula eliminated by $\opa
E$ is not the principal formula of an $\Intro\opa$ rule eliminated by the
$\Elim\opa$ rule.  The first step always works if the premises of any
$\Intro\opa$ and $\Elim\opa$ inference are jointly unsatisfiable clauses, as
is the case with the introduction and general elimination rules
constructed by our general method.  The second step also always works,
as long as the calculus has the general form of $\Nms X$, although not
all of its features are essential.  For instance, we can modify the
proof to work on a calculus with sequences instead of multisets of
formulas but with exchange rules, simply by not counting exchange
inferences in the length of segments.  However, we run into problems
if left contraction is explicitly present. In
conversion reductions, contractions are used to simulate resolution
inferences by cuts. These cannot in general be avoided, even if the
resolution proof is Horn.  For instance, consider
\[
  \Axiom$\Gamma \fCenter A$ 
  \Axiom$A, A, \Pi \fCenter \Lambda$
  \RightLabel{$\Left\Contr$}
  \UnaryInf$A, \Pi \fCenter \Lambda$
\RightLabel{$\cut$}
\BinaryInf$\Gamma, \Pi \fCenter \Delta, \Lambda$
\DisplayProof
\]
Although such a cut can be replaced by two cuts, it multiplies the
context formulas:
\[
  \Axiom$\Gamma \fCenter A$ 
    \Axiom$\Gamma \fCenter A$
    \Axiom$A, \Pi \fCenter \Lambda$
  \RightLabel{$\cut$}
  \BinaryInf$A, \Gamma, \Pi \fCenter \Delta, \Lambda$
\RightLabel{$\cut$}
\BinaryInf$\Gamma, \Gamma, \Pi \fCenter \Delta, \Delta, \Lambda$
\DisplayProof
\]
To obtain the original conclusion, we would need contraction
again. One might consider replacing the cut rule with a multi-cut or
``mix'' rule that allows the removal of any number of occurrences of
the cut formula to avoid the difficulty.  Such a rule, however, allows
us to simulate contraction (by applying mix to a suitable initial
sequent), and so nothing is gained.  In natural deduction where
assumptions are labelled, these problems are avoided since the work of
contractions on the right is done by having assumptions in different
initial sequents sharing a label.

\begin{prop}
A calculus $\Nmp X$ resulting from $\Nmsl X$ by replacing any general
elimination rules by specialized elimination rules also normalizes.
\end{prop}

\begin{proof}
To verify that segments of length~1 can always be removed, we have to
establish that there are always simplification conversions for $\opa
I$ rules followed immediately by a specialized $\Elim\opa'$ rule.  The
set of clauses corresponding to the premises of the original (or
derived) general elimination rule $\Elim\opa$ together with the clauses
corresponding to the premises of the $\Intro\opa$ rule are unsatisfiable.
We again obtain a derivation using only cuts and structural rules of
the conclusion of the $\Elim\opa'$ inference by adding the context
formulas $\Gamma_i$, $\Delta_i$ to the clauses corresponding to the
premises of $\Intro\opa$, and adding $\Gamma_j$, $\Delta_j$ to the clauses
corresponding to the minor premises of $\Elim\opa'$.  The clauses
corresponding to the missing minor premises are all of the form $x
\colon A \fCenter$.  Add $A$ to its right side and every sequent
inferred from it. The clause thus turns into an initial sequent; $A$
remains present in the succedent of any sequent derived from $x
\colon A \fCenter A$, thus also in the last sequent of the derivation
fragment.

One easily verifies that $\Elim\opa'$ rules also permute across other
rules.
\end{proof}

\section{Single-conclusion sequent calculi}
\label{sec:sclk}

From the multiple conclusion sequent calculus $\LK$ we obtain a
single-conclusion system $\LJ$ by restricting the succedent in every
rule and every sequent in a proof to contain at most one formula.  The
calculus $\LJ$ is sound and complete for intuitionistic logic.  This
idea can be applied to any calculus $\L X$, of course, and results in
an ``intuitionistic'' variant of the calculus.

We begin by considering the sequent calculus $\Ls X$ resulting from
$\L X$ by restricting the rules in such a way that the succedent is
guaranteed to contain at most one formula. This requires first of all
that the premises of each rule have at most one auxiliary formula in
the succedent. This can always be achieved by replacing a rule that
does not satisfy this condition by split rules that do (see
Corollary~\ref{intuitionistic-rules}).  We will assume that the
logical rules of $\L X$ do satisfy this condition.

The restriction in $\Ls X$ requires that in any application of a rule,
the succedent of the premises and conclusion contains at most one
formula.  This means that in every right rule, $\Delta$ is empty.
Furthermore, if a premise of a left rule has an auxiliary formula
$\Lambda_i$ in the succedent, $\Delta$ must also be empty in that
premise (i.e., there is no side formula). If all premises of a left
rule have an auxiliary formula on the right, then no premise allows a
side formula~$\Delta$, and the succedent of the conclusion of the
rule is empty. We'll call rules of this form \emph{restricted}.

The $\Right\Weak$ and $\cut$ rules now take the form:
\[
\Axiom$\Gamma \fCenter $
\RightLabel{$\Right\Weak$}
\UnaryInf$\Gamma \fCenter C$
\DisplayProof
 \qquad
\RightLabel{$\cut$}
  \Axiom$\Gamma \fCenter A$
  \Axiom$A, \Pi \fCenter \Lambda$
\BinaryInf$\Gamma, \Pi \fCenter \Lambda$
\DisplayProof
\]
where $\Lambda$ may contain at most one formula.

The intuitionistic sequent calculus~$\LJ$ is the single-conclusion
sequent calculus corresponding to the multi-conclusion classical
sequent calculus~$\LK$ obtained in this way. 

\begin{prop}
If $\Ls X$ proves $\Gamma \fCenter \Delta$ (and $\Delta$ thus contains
at most one formula), then $\L X$ does as well.  The converse does not
hold in general.
\end{prop}

\begin{proof}
Proofs in $\Ls X$ can be translated into proofs in $\L X$ directly,
adding weakenings to provide the missing shared side formulas~$\Delta$
in succedents of left rules~$\L X$ if necessary.  A counterexample to
the converse is given by ${} \fCenter A \lor \lnot A$, which is provable
in $\LK$ but not in~$\LJ$.
\end{proof}

Not every set of restricted rules will result in a reasonable
single-conclusion sequent calculus. For instance, consider the
restricted rules for $\nif$, the negated conditional (aka
``exclusion'') in which the left and right rules for $\lif$ are
reversed:
\[
  \Axiom$A, \Gamma \fCenter B$
  \RightLabel{$\Left\nif$}
  \UnaryInf$A \nif B, \Gamma \fCenter$
  \DisplayProof
 \qquad
    \Axiom$\Gamma \fCenter A$
    \Axiom$B, \Gamma \fCenter$
  \RightLabel{$\Right\nif$}
  \BinaryInf$\Gamma \fCenter A \nif B$
\DisplayProof
\]
Note that the succedent is required to be empty in the right premise
of $\Right\nif$ and in the conclusion of $\Left\nif$. Because of this, $A
\nif B \fCenter A \nif B$ cannot be derived from $A \fCenter A$ and $B
\fCenter B$. In the regular sequent calculus for $\nif$ we would have:
\[
        \Axiom$A \fCenter A$
        \doubleLine
        \Axiom$B \fCenter B$
        \UnaryInf$B, A \fCenter B$ 
      \RightLabel{$\Right\nif$}
      \BinaryInf$A \fCenter B, A \nif B$
    \doubleLine
    \UnaryInf$A \fCenter A \nif B, B$
  \RightLabel{$\Left\nif$}
  \UnaryInf$A \nif B \fCenter A \nif B$
  \DisplayProof
  \qquad
      \Axiom$A \fCenter A$
      \doubleLine
      \UnaryInf$A \fCenter A, B$
    \RightLabel{$\Left\nif$}
    \UnaryInf$A \nif B \fCenter A$
        \Axiom$B \fCenter B$
        \doubleLine
        \UnaryInf$A, B \fCenter B$
      \RightLabel{$\Left\nif$}
      \UnaryInf$A \nif B, B \fCenter$
    \doubleLine
    \UnaryInf$B, A \nif B \fCenter$
  \RightLabel{$\Right\nif$}
  \BinaryInf$A \nif B \fCenter A \nif B$
\DisplayProof
\]
These are not correct derivations in the restricted
calculus.\footnote{The inability of a set of rules to derive
$\opa(\vec A) \fCenter \opa(\vec A)$ from $A_i \fCenter A_i$ is of course
not a proof that the rules are incomplete, especially since we do not
have a semantics with respect to which the question can be posed.
However, assuming the semantics validate substitution of equivalent
formulas and there are pairs of equivalent formulas, cut-free
incompleteness follows. For instance, it will be impossible to give a
cut-free proof of $\opa(A) \fCenter \opa(A \land A)$ without deriving it
from sequents $A \fCenter A$, $A \fCenter A \land A$, and $A \land A
\fCenter A$ using only the rules for $\opa$, whatever they may be.} To
obtain a set of rules in which it is possible to give such a
derivation, we must split the $\Left\nif$ rule further to guarantee that
no premise has auxiliary formulas on both the left and the right side:
\[
    \Axiom$A, \Gamma \fCenter \Delta$
  \RightLabel{$\Left\nif_1$}
  \UnaryInf$A \nif B, \Gamma \fCenter \Delta$
\DisplayProof
\qquad
    \Axiom$\Gamma \fCenter B$
  \RightLabel{$\Left\nif_2$}
  \UnaryInf$A \nif B, \Gamma \fCenter $
\DisplayProof
\]
Now the derivation can be carried out:
\[
      \Axiom$A \fCenter A$
    \RightLabel{$\Left\nif_1$}
    \UnaryInf$A \nif B \fCenter A$
        \Axiom$B \fCenter B$
        \RightLabel{$\Left\nif_2$}
        \UnaryInf$A \nif B, B \fCenter$
    \doubleLine
    \UnaryInf$B, A \nif B \fCenter $
  \RightLabel{$\Right\nif$}
  \BinaryInf$A \nif B \fCenter A \nif B$
\DisplayProof
\]
(This corresponds to the right of the two derivations above; a version
corresponding to the left one, where we first apply $\Right\nif$ and then
$\Left\nif$ is not possible since the right premise of the restricted
$\Right\nif$ must have empty succedent.)

Perhaps surprisingly, insufficient intuitionistic calculi can also
result from splitting rules too much. Consider the restricted rules
for \textsc{nand}, i.e., the Sheffer stroke:
\[
  \Axiom$A, B, \Gamma \fCenter$
\RightLabel{$\Right\mid$}
\UnaryInf$\Gamma \fCenter A \mid B$
\DisplayProof
\qquad
  \Axiom$\Gamma \fCenter A$
  \Axiom$\Gamma \fCenter B$
\RightLabel{$\Left\mid$}
\BinaryInf$A \mid B, \Gamma \fCenter$
\DisplayProof
\]
These can derive $A \mid B \fCenter A \mid B$ from atomic sequents:
\[
        \Axiom$A \fCenter A$
      \doubleLine
      \UnaryInf$A, B \fCenter A$
        \Axiom$B \fCenter B$
      \doubleLine
      \UnaryInf$A, B \fCenter B$
    \RightLabel{$\Left\mid$}
    \BinaryInf$A \mid B, A, B \fCenter$
  \doubleLine
  \UnaryInf$A, B, A \mid B \fCenter$
\RightLabel{$\Right\mid$}
\UnaryInf$A \mid B \fCenter A \mid B$
\DisplayProof
\]
However, this is not possible when the $\Right\mid$ rule is split into
\[
  \bottomAlignProof
  \Axiom$A, \Gamma \fCenter$
\RightLabel{$\Right\mid_1$}
\UnaryInf$\Gamma \fCenter A \mid B$
\DisplayProof
\qquad\text{and}\qquad
  \bottomAlignProof
  \Axiom$B, \Gamma \fCenter$
\RightLabel{$\Right\mid_2$}
\UnaryInf$\Gamma \fCenter A \mid B$
\DisplayProof
\]
In the unrestricted calculus, the last inference of the above proof
can be replaced by
\[
        \Axiom$A, B, A \mid B \fCenter$
      \RightLabel{$\Right\mid_1$}
      \UnaryInf$B, A \mid B \fCenter A \mid B$
    \RightLabel{$\Right\mid_2$}
    \UnaryInf$A \mid B \fCenter A \mid B, A \mid B$
  \RightLabel{$\Right\Contr$}
  \UnaryInf$A \mid B \fCenter A \mid B$
\DisplayProof
\]
The application of $\Right\mid_2$ is not allowed in the restricted
calculus, since there the right side of the premise is restricted to
be empty.

\begin{table}
  \caption{Single conclusion sequent rules for some unusual connectives}
\[
\begin{array}{ccc}
\hline\hline
\text{connective} & \text{rules}\\ \hline\\ 
A \nif B & 
  \Axiom$\Gamma \fCenter A$
  \Axiom$B, \Gamma \fCenter$
  \RightLabel{$\Right\nif$}
\BinaryInf$\Gamma \fCenter A \nif B$
\DisplayProof
\\[2ex]
\text{(exclusion)} & 
  \Axiom$\Gamma \fCenter B$ 
\RightLabel{$\Left\nif_1$}
\UnaryInf$A \nif B, \Gamma \fCenter$
\DisplayProof
& 
\Axiom$A, \Gamma \fCenter \Delta$
\RightLabel{$\Left\nif_2$}
\UnaryInf$A \nif B, \Gamma \fCenter \Delta$
\DisplayProof \\[2ex] 
A \mid B \text{ (nand)} & 
\Axiom$A, B, \Gamma \fCenter$
\RightLabel{$\Right\mid$}
\UnaryInf$\Gamma \fCenter A \mid B$
\DisplayProof
 & 
\Axiom$\Gamma \fCenter A$
\Axiom$\Gamma \fCenter B$
\RightLabel{$\Left\mid$}
\BinaryInf$A \mid B, \Gamma \fCenter $
\DisplayProof 
\\[2ex]
A \nor B \text{ (nor)} 
&  
\Axiom$A, \Gamma \fCenter$
\Axiom$B, \Gamma \fCenter$
\RightLabel{$\Right\nor$}
\BinaryInf$\Gamma \fCenter A \nor B$
\DisplayProof 
\\[2ex]
& 
  \Axiom$\Gamma \fCenter A$
\RightLabel{$\Left\nor_1$}
\UnaryInf$A \nor B, \Gamma \fCenter$
\DisplayProof
& 
\Axiom$\Gamma \fCenter B$
\RightLabel{$\Left\nor_2$}
\UnaryInf$A \nor B, \Gamma \fCenter $
\DisplayProof \\[2ex]
A \xor B & 
\Axiom$\Gamma \fCenter A$
\Axiom$A, B, \Gamma \fCenter$
\RightLabel{$\Right\xor_1$}
\BinaryInf$\Gamma \fCenter A \xor B$
\DisplayProof 
 & 
   \Axiom$\Gamma \fCenter B$
   \Axiom$A, B, \Gamma \fCenter $
\RightLabel{$\Right\xor_2$}
\BinaryInf$\Gamma \fCenter A \xor B$
\DisplayProof 
\\[2ex]
\text{(xor)} 
&  
  \Axiom$A, \Gamma \fCenter \Delta$
  \Axiom$B, \Gamma \fCenter \Delta$ 
\RightLabel{$\Left\xor_1$}
\BinaryInf$A \xor B, \Gamma \fCenter \Delta$
\DisplayProof
&  
   \Axiom$\Gamma \fCenter B$
   \Axiom$\Gamma \fCenter A$
\RightLabel{$\Left\xor_2$}
\BinaryInf$A \xor B, \Gamma \fCenter $
\DisplayProof 
\\[2ex]
A \lif B / C 
& 
    \Axiom$A, \Gamma \fCenter B$
    \Axiom$\Gamma \fCenter A$
\RightLabel{$\Right\ite_1$}
\BinaryInf$\Gamma \fCenter A \lif B/C$
\DisplayProof
 & 
  \Axiom$A, \Gamma \fCenter B$
  \Axiom$\Gamma \fCenter C$
  \RightLabel{$\Right\ite_2$}
\BinaryInf$\Gamma \fCenter A \lif B/C$
\DisplayProof
\\[2ex]
      \text{(if then else)} 
& 
  \Axiom$A, B, \Gamma \fCenter \Delta$
  \Axiom$\Gamma \fCenter A$
\RightLabel{$\Left\ite_1$}
\BinaryInf$A \lif B/C, \Gamma \fCenter \Delta$
\DisplayProof
& 
  \Axiom$A, B, \Gamma \fCenter \Delta$
  \Axiom$C, \Gamma \fCenter \Delta$
\RightLabel{$\Left\ite_2$}
\BinaryInf$A \lif B/C, \Gamma \fCenter \Delta$
\DisplayProof\\[2ex]
  \hline\hline
\end{array}
\]
\end{table}

The question of when suitable restricted calculi $\Ls X$ are sound and
complete for (something like) an intuitionistic semantics will be the
topic of a future paper (but see \cite{BaazFermuller1996} and
\cite{GeuversHurkens2017} for results in this direction). However,
any calculus the rules of which satisfy the restriction, also has
cut elimination.

\begin{prop}
The cut elimination theorem holds for $\Ls X$.
\end{prop}

\begin{proof}
The rules for $\Ls X$ are obtained, if necessary, by first splitting
rules to guarantee that the premises of of the rule contain at most
one auxiliary formula on the right.  In this case, the clauses
corresponding to the premises of a rule for a
connective are Horn, i.e., they contain at most one positive literal.
Resolution on Horn clauses only produces Horn clauses.  Since the set
of clauses corresponding to the premises of a pair of $\Left\opa$ and
$\Right\opa$ rules is unsatisfiable (even when these are split rules), it
has a resolution refutation consisting only of Horn clauses.  As
before, we can add the side formulas of the actual premises to this
resolution refutation to obtain a derivation of the conclusion of the
\mix{} from the premises of the $\Left\opa$ and $\Right\opa$ rules. Since
each premise has at most one formula in the succedent, \mix{}
inferences cannot result in more than one formula on the right in the
resulting proof segment. Thus the \mix{} on $\opa(\vec A)$ can be
replaced with a sequence of \mix es with the subformulas~$A_i$ as
\mix{} formulas in~$\Ls X$.

When permuting \mix{} inferences with rules to reduce the rank, we
have to verify that the resulting inferences obey the restrictions of
the rules of $\Ls X$. First of all, observe that if the inference
ending in the left premise of the \mix{} is $\Right\opa$, the mix
formula must be $\opa(\vec A)$ and thus the left rank is~$1$. Thus we
never have to permute a \mix{} with an $\Right\opa$ rule on the left
side of a~\mix.

If the last inference on the right side of the \mix{} is $\opa
L$ introducing the \mix{} formula, we start from a derivation that has
the form
\[
    \Axiom$\Gamma \fCenter \opa(\vec A)$
    \Axiom$\cdots\quad \Pi_i', \Theta \fCenter \Xi_i, \Lambda_i' \quad\cdots$
    \RightLabel{$\Left\opa$}
    \UnaryInf$\opa(\vec A), \Theta \fCenter \Xi$
  \RightLabel{$\mix: \opa(\vec A)$}
  \BinaryInf$\Gamma, \Theta^* \fCenter \Delta^*, \Xi$
\DisplayProof
\]
where $\Xi_i$ is empty if $\Lambda_i'$ is not, and is $\Xi$ otherwise.
$\Xi$ itself is empty if no $\Lambda_i'$ is empty. This guarantees
that $\Xi$ and $\Xi_i, \Lambda'_i$ all contain at most one formula. 
This is converted to
\[
    \Axiom$\Gamma \fCenter \opa(\vec A)$
    \Axiom$\Pi_i', \Theta \fCenter \Xi_i, \Lambda_i'$
\RightLabel{$\mix: \opa(\vec A)$}
\BinaryInf$\Gamma, \Pi_i', \Theta^* \fCenter \Xi_i, \Lambda_i'$
\DisplayProof
\]
Since by induction hypothesis, this \mix{} can be removed, we get
\mix-free derivations of $\Pi_i', \Gamma, \Theta^* \fCenter \Xi_i,
\Lambda_i'$ (by applying suitable $\Left\Exch$ inferences). These are
exactly the premises of a correct $\Left\opa$ inference, as required.
Now consider
\[
    \Axiom$\Gamma \fCenter \opa(\vec A)$
    \Axiom$\cdots\quad \Pi_i', \Gamma, \Theta^* \fCenter \Xi_i, \Lambda_i' \quad\cdots$
    \RightLabel{$\Left\opa$}
    \UnaryInf$\opa(\vec A), \Gamma, \Theta^* \fCenter \Xi$
\RightLabel{$\mix: \opa(\vec A)$}
\BinaryInf$\Gamma, \Pi_i', \Theta^* \fCenter \Xi_i, \Lambda_i'$
\DisplayProof
\]
As $\opa(\vec A)$ does not occur in $\Pi_i', \Gamma, \Theta^*$, the
right rank is now~$1$, and by induction hypothesis the \mix{} can be
eliminated. 

The other interesting cases where the right premise is the conclusion
of $\Left\opa$ not introducing the \mix{} formula or of $\Right\opa$ are
similar. 
\end{proof}

As an example of reducing the degree of a \mix{} on a formula
introduced by restricted left and right rules other than the usual
ones in $\LJ$, consider the calculus for the Sheffer stroke from before.
A \mix{} of rank~$2$ on $A \mid B$,
\[
  \Axiom$A, B, \Gamma \fCenter$
  \RightLabel{$\Right\mid$}
  \UnaryInf$\Gamma \fCenter A \mid B$
      \Axiom$\Theta \fCenter A$
      \Axiom$\Theta \fCenter B$
    \RightLabel{$\Left\mid$}
    \BinaryInf$A \mid B, \Theta \fCenter $
  \RightLabel{$\mix: A \mid B$}
  \BinaryInf$\Gamma, \Theta \fCenter $
\DisplayProof
\]
is reduced to \mix{} inferences on $A$ and $B$:
\[
  \Axiom$\Theta \fCenter B$
    \Axiom$\Theta \fCenter A$
    \Axiom$A, B, \Gamma \fCenter $
  \RightLabel{$\mix: A$}
  \BinaryInf$\Theta, B, \Gamma \fCenter$
\RightLabel{$\mix: B$}
\BinaryInf$\Theta, \Theta, \Gamma \fCenter$
\DisplayProof
\]
This corresponds to the resolution refutation:
\[
\Axiom$\fCenter B$
  \Axiom$\fCenter A$
  \Axiom$A, B \fCenter$
  \BinaryInf$B \fCenter$
\BinaryInf$\fCenter$
\DisplayProof
\]

\section{Classical single-conclusion sequent calculi}
\label{sec:classical}

It is possible to turn an intuitionistic, single-conclusion
sequent calculus into a classical one without allowing multiple
formulas in the succedent. The simplest way to do this is to introduce
additional initial sequents, e.g., $\fCenter A \lor \lnot A$ or
$\lnot\lnot A \fCenter A$. In natural deduction, a classical system can
also be obtained from $\NJ$ by adding axioms, but it is more, well,
\emph{natural} to add a rule instead. \cite{Prawitz1965} proposed the
classical absurdity rule
\[
  \Axiom$\lnot A, \Gamma \fCenter \bot$
\RightLabel{$\bot_C$}
\UnaryInf$\Gamma \fCenter A$
\DisplayProof
\]
For the sequent calculus, the corresponding rule would replace $\bot$
in the conclusion with an empty succedent:
\[
  \Axiom$\lnot A, \Gamma \fCenter$
\RightLabel{$\bot_C$}
\UnaryInf$\Gamma \fCenter A$
\DisplayProof
\]
Equivalent rules are double negation elimination and rule of excluded
middle:
\[
  \Axiom$A, \Gamma \fCenter C$
  \Axiom$\lnot A, \Gamma \fCenter C$
\RightLabel{$\gem$}
\BinaryInf$\Gamma \fCenter C$
\DisplayProof
\]
These rules, however, do not have the subformula property in at least the
extended sense that $\bot_C$ does, where every formula in the premise
is a subformula of a formula in the conclusion or the negation of one.

Suppose now that something like the negation connective is present in
$X$ and that $\Ls X$ has the usual $\Left\lnot$ and $\Right\lnot$ rules. It
suffices in fact that a connective that behaves like $\lnot$ can be
expressed with the connectives of~$X$ in such a way that the $\Left\lnot$
and $\Right\lnot$ rules can be simulated. E.g., if the Sheffer stroke is
present, the corresponding version of the $\bot_C$ rule would be
\[
\Axiom$A \mid A, \Gamma \fCenter $
\RightLabel{$\bot_C$}
\UnaryInf$\Gamma \fCenter A$
\DisplayProof
\]

\begin{prop}\label{trans-LXs}
If $\L X$ proves $\Gamma \fCenter \Delta$ then $\Ls X + \bot_C$ proves
$\Gamma \fCenter \Delta$ (if $\Delta$ contains at most one formula).
\end{prop}

\begin{proof}
We define a translation of proofs in $\L X$ to proofs in $\Ls X +
\bot_C$ by induction on the height of a proof in $\L X$. The
end-sequent of a translation of a proof of $\Gamma \fCenter \Delta, D$
is $\Gamma, \Delta^\lnot \fCenter D$, where $\Delta^\lnot$ is $\lnot
A_n, \dots, \lnot A_1$ if $\Delta$ is $A_1, \dots, A_n$. The
translation of a proof of $\Gamma \fCenter {}$ also ends in $\Gamma
\fCenter {}$.

The translation of an initial sequent $A \fCenter A$ is just $A \fCenter A$.

If the proof ends in $\Right\Weak$, either add $\Right\Weak$ if the
succedent of the premise is empty, or $\Left\Weak$ on the negation of
the weakening formula plus $\Left\Exch$.

If the proof ends in $\Right\Exch$ where the exchanged formulas do not include
the rightmost formula, add suitable $\Left\Exch$ inferences to the
translation of the proof of the premise.

A proof ending in $\Right\Contr$ is translated as follows:
\[
\Axiom$\Gamma \fCenter \Delta, D, D$
\RightLabel{$\Right\Contr$}
\UnaryInf$\Gamma \fCenter \Delta, D$
\DisplayProof
\qquad
        \Axiom$\Gamma, \lnot D, \Delta^\lnot \fCenter D$
      \RightLabel{$\Left\lnot$}
      \UnaryInf$\lnot D, \Gamma, \lnot D, \Delta^\lnot \fCenter$
    \doubleLine
    \RightLabel{$\Left\Exch$}
    \UnaryInf$\lnot D, \lnot D, \Gamma, \Delta^\lnot \fCenter $
  \RightLabel{$\Left\Contr$}
  \UnaryInf$\lnot D, \Gamma, \Delta^\lnot \fCenter $
\RightLabel{$\bot_C$}
\UnaryInf$\Gamma, \Delta^\lnot \fCenter D$
\DisplayProof
\]
A proof ending in $\Right\Exch$ in which the rightmost formula is active is
translated as follows:
\[
\Axiom$\Gamma \fCenter \Delta, D, C$
\RightLabel{$\Right\Exch$}
\UnaryInf$\Gamma \fCenter \Delta, C, D$
\DisplayProof
\qquad
      \Axiom$\Gamma, \lnot D, \Delta^\lnot \fCenter C$
    \RightLabel{$\Left\lnot$}
    \UnaryInf$\lnot C, \Gamma, \lnot D, \Delta^\lnot \fCenter $
  \doubleLine
  \RightLabel{$\Left\Exch$}
  \UnaryInf$\lnot D, \Gamma, \lnot C, \Delta^\lnot \fCenter $
\RightLabel{$\bot_C$}
\UnaryInf$\Gamma, \lnot C, \Delta^\lnot \fCenter D$
\DisplayProof
\]
If the proof ends in a cut, in a weakening or contraction on the left,
or in a logical inference, add the corresponding inferences to the
translations of the proofs ending in the premise(s).
\end{proof}

The converse of course also holds, since every application of a rule
of $\Ls X$ is also a correct application of a rule of $\L X$, and
$\bot_C$ can be derived in $\L X$ by
\[
  \Axiom$A \fCenter A$
  \RightLabel{$\Right\lnot$}
  \UnaryInf$\fCenter A, \lnot A$
  \Axiom$\lnot A, \Gamma \fCenter $
\RightLabel{$\cut$}
\BinaryInf$\Gamma \fCenter A$
\DisplayProof
\]

In order to obtain cut elimination results for $\Ls X$, it is
convenient to replace the $\bot_C$ rule with a classical version of
the \cut{} rule:
\[
  \Axiom$\lnot A, \Gamma \fCenter $
  \Axiom$A, \Pi \fCenter \Lambda$
\RightLabel{$\kut$}
\BinaryInf$\Gamma, \Pi \fCenter \Lambda$
\DisplayProof
\]
($\Lambda$ contains at most one formula.)

The $\kut$ rule can simulate the $\bot_C$ rule over $\Ls X$:
\[
  \Axiom$\lnot A, \Gamma \fCenter$
  \RightLabel{$\bot_C$}
  \UnaryInf$\Gamma \fCenter A$
  \DisplayProof
  \qquad
  \Axiom$\lnot A, \Gamma \fCenter$
  \Axiom$A \fCenter A$
  \RightLabel{$\kut$}
  \BinaryInf$\Gamma \fCenter A$
  \DisplayProof
\]
In the reverse direction, $\bot_C$ together with \cut{} can simulate \kut:
\[
\Axiom$\lnot A, \Gamma \fCenter $
\Axiom$A, \Pi \fCenter \Lambda$
\RightLabel{$\kut$}
\BinaryInf$\Gamma, \Pi \fCenter \Lambda$
\DisplayProof
\qquad
  \Axiom$\lnot A, \Gamma \fCenter $
  \RightLabel{$\bot_C$}
  \UnaryInf$\Gamma \fCenter A$ 
  \Axiom$A, \Pi \fCenter \Lambda$
\RightLabel{$\cut$}
\BinaryInf$\Gamma, \Pi \fCenter \Lambda$
\DisplayProof
\]
(Again, $\Lambda$ contains at most one formula.)

Consequently, the previous result establishing structure-preserving
translations of $\L X$ proofs into $\Ls X + \bot_C$ proofs also
transfers to $\Ls X + \kut$ proofs.  As an example, consider the
derivation of excluded middle in $\LJ + \kut$ given by
\[
              \Axiom$A \fCenter A$
              \RightLabel{$\Right\lor_1$}
              \UnaryInf$A \fCenter A \lor \lnot A$
            \RightLabel{$\Left\lnot$}
            \UnaryInf$\lnot(A \lor \lnot A), A \fCenter $
          \RightLabel{$\Left\Exch$}
          \UnaryInf$A, \lnot(A \lor \lnot A) \fCenter $
        \RightLabel{$\Right\lnot$}
        \UnaryInf$\lnot(A \lor \lnot A) \fCenter \lnot A$
      \RightLabel{$\Right\lor_2$}
      \UnaryInf$\lnot(A \lor \lnot A) \fCenter A \lor \lnot A$
    \RightLabel{$\Left\lnot$}
    \UnaryInf$\lnot(A \lor \lnot A), \lnot(A \lor \lnot A) \fCenter $
  \RightLabel{$\Left\Contr$}
  \UnaryInf$\lnot(A \lor \lnot A) \fCenter $
  \Axiom$A \lor \lnot A \fCenter A \lor \lnot A$
\RightLabel{$\kut$}
\BinaryInf$\fCenter A \lor \lnot A$
\DisplayProof
\]

The $\gem$ rule is also derivable using $\kut$:
\[
      \Axiom$\lnot A, \Gamma \fCenter C$
    \RightLabel{$\Left\lnot$}
    \UnaryInf$\lnot C, \lnot A, \Gamma \fCenter $
  \RightLabel{$\Left\Exch$}
  \UnaryInf$\lnot A, \lnot C, \Gamma \fCenter $
  \Axiom$A, \Pi \fCenter C$
\RightLabel{$\kut$}
\BinaryInf$\lnot C, \Gamma, \Pi \fCenter C$
\RightLabel{$\Left\lnot$}
\UnaryInf$\lnot C, \lnot C, \Gamma, \Pi \fCenter$
\RightLabel{$\Left\Contr$}
\UnaryInf$\lnot C, \Gamma, \Pi \fCenter $
  \Axiom$C \fCenter C$
\RightLabel{$\kut$}
\BinaryInf$\Gamma, \Pi \fCenter C$
\DisplayProof
\]

The restricted system $\Ls X$ enjoys cut elimination in precisely the
same way $\LJ$ does.  However, for the corresponding classical systems
$\Ls X + \bot_C$ and $\Ls X + \kut$ the situation is more complicated.
Clearly, $\bot_C$ and $\kut$, considered as variant cut rules, cannot
be eliminated from proofs in $\Ls X$. We might hope, however, that
$\cut$ can be eliminated from proofs in $\Ls X + \bot_C$ and $\Ls X +
\kut$, and that we can ``control'' the applications of $\bot_C$ and
$\kut$, e.g., to atomic~$A$. This was shown by \cite{NegrivonPlato2001}
to hold for their related system $\mathbf{G3ip} + \gem$, via the
Sch\"utte-Tait method of cut elimination.

Like $\cut$, neither $\bot_C$ nor $\kut$ permute across $\Right\Contr$.  In
order to study Gentzen's cut-elimination procedure for $\Ls X$ we
should thus consider their analogs to Gentzen's $\mix$ rule which
allows the removal of any number of occurrences of the cut
formula~$A$, not just the outermost ones.  Let us call these rules
$\bot_C^*$ and $\kix$.

It is, however, possible to obtain restricted cut elimination results
for $\kut$.

\begin{prop}
$\kix$ permutes with $\mix$.
\end{prop}

\begin{proof}
We give the derivations for cases where the rules are also
applications of $\kut$ and $\cut$ for simplicity.
\[
    \Axiom$\lnot A, \Gamma \fCenter $
    \Axiom$A, \Pi \fCenter C$
  \RightLabel{$\kix$}
  \BinaryInf$\Gamma, \Pi \fCenter C$
  \Axiom$\Pi' \fCenter \Lambda'$
\RightLabel{$\mix$}
\BinaryInf$\Gamma, \Pi, \Pi' \fCenter \Lambda'$
\DisplayProof
\qquad
  \Axiom$\lnot A, \Gamma \fCenter $
    \Axiom$A, \Pi \fCenter C $
    \Axiom$C, \Pi' \fCenter \Lambda'$
  \RightLabel{$\mix$}
  \BinaryInf$A, \Pi, \Pi' \fCenter \Lambda'$
\RightLabel{$\kix$}
\BinaryInf$\Gamma, \Pi, \Pi' \fCenter \Lambda'$
\DisplayProof
\]
\[
  \Axiom$\Gamma \fCenter C $
    \Axiom$\lnot A, C, \Pi \fCenter$
    \Axiom$A, \Pi' \fCenter \Lambda'$
  \RightLabel{$\kix$}
  \BinaryInf$C, \Pi, \Pi' \fCenter \Lambda'$
\RightLabel{$\mix$}
\BinaryInf$\Gamma, \Pi, \Pi' \fCenter \Lambda'$
\DisplayProof
\qquad
    \Axiom$\Gamma \fCenter C$ 
    \Axiom$\lnot A, C, \Pi \fCenter$
    \doubleLine
    \UnaryInf$C, \lnot A, \Gamma \fCenter $ 
  \RightLabel{$\mix$}
  \BinaryInf$\lnot A, \Gamma, \Pi \fCenter $
  \Axiom$A, \Pi' \fCenter  \Lambda'$
\RightLabel{$\kix$}
\BinaryInf$\Gamma, \Pi, \Pi' \fCenter \Lambda'$
\DisplayProof
\]
Note that if $C$ is $\lnot A$ in the second case, the starting
derivation is impossible, since the conclusion of the \mix{} then does
not contain $C$ on the left.
\end{proof}

By contrast, $\bot_C^*$ (and thus also $\bot_C$) does not permute with
$\mix$. Consider the following case:
\[
  \Axiom$\lnot A, \Gamma \fCenter $
  \RightLabel{$\bot_C^*$}
  \UnaryInf$\Gamma \fCenter A$
  \Axiom$A, \Pi \fCenter C$
\RightLabel{$\mix$}
\BinaryInf$\Gamma, \Pi \fCenter C$
\DisplayProof
\]
Since the cut formula~$A$ does not appear in the succedent of the
premise of the $\bot_C^*$ rule, we cannot apply the $\mix$ rule to it.
If we first use $\Right\Weak$ to introduce the cut formula $A$ in the
succedent, the conclusion sequent still contains $\lnot A$ on the
left, but application of $\bot_C^*$ is blocked by the presence of $C$
in the succedent.  It would be possible to derive from $A, \Pi \fCenter
C$ the sequent $\lnot C, \Pi \fCenter \lnot A$ (using $\Left\lnot$ and
$\Right\lnot$) and then apply a \mix{} on $\lnot A$, but this would
increase the degree of the \mix{} formula and (because of the
additional inferences required) would not be guaranteed to decrease
the rank of the \mix.

One last strategy to avoid this difficulty would be to show that we
can transform the proof of the premise of $\bot_C^*$ into one of its
conclusion that avoids the $\bot_C^*$ inference and does not increase
the height of that subproof. Then the application of $\mix$ would
apply to two subproofs of lower height than the original (since the
$\bot_C^*$ rule would be removed) and so the induction hypothesis
would apply.  However, the conclusion $\bot_C^*$ rule is in general
not provable without $\bot_C^*$ \emph{at all}. Consider the case
$\lnot(A \lor \lnot A) \fCenter {}$, which is derivable without
$\bot_C$. However, the corresponding conclusion of $\bot_C$, $\fCenter A
\lor \lnot A$, is not so derivable.

\begin{prop}
$\kix$ can be replaced by $\mix$ if the cut formula~$\lnot A$ is
  principal in the left premise.
\end{prop}

\begin{proof}
If $\lnot A$ is principal, it was introduced by a $\Left\lnot$ rule. Then
we can replace
\begin{endproofeqnarray*}
  \bottomAlignProof
  \Axiom$\Gamma \fCenter A$
  \RightLabel{$\Left\lnot$}
  \UnaryInf$\lnot A, \Gamma \fCenter$ 
  \Axiom$A, \Pi \fCenter \Lambda$
\RightLabel{$\kix$}
\BinaryInf$\Gamma, \Pi \fCenter \Lambda$
\DisplayProof
 &\quad\text{by}\quad
 &
 \bottomAlignProof
\Axiom$\Gamma \fCenter A$
\Axiom$A, \Pi \fCenter \Lambda$
\RightLabel{$\mix$}
\BinaryInf$\Gamma, \Pi \fCenter \Lambda$
\DisplayProof
\end{endproofeqnarray*}
\end{proof}

\section{Single-conclusion natural deduction}
\label{sec:scnd}

We have seen that the multi-conclusion ``sequent style'' natural
deduction systems $\Nmsl X$ always normalize.  If we restrict the
succedent of sequents in derivations to at most one formula, we
obtain a single-conclusion ``sequent style'' natural deduction
system~$\Ns X$.  This will be an ``intuitionistic'' version of $\Nmsl
X$.  For the standard set of logical operators, this system is the
standard intuitionistic natural deduction system in sequent style,
with one slight difference. In the standard systems, the succedent of
sequents always contains exactly one formula; an empty succedent is
represented by the contradiction symbol~$\bot$. So, if we are content
to add the contradiction constant $\bot$ to the language, we can obtain
the standard systems simply by marking every empty succedent in the
rules of $\Ns X$ with~$\bot$.

An example rule of an intuitionistic sequent-style natural deduction
rule would be
\[
  \Axiom$\Gamma_0 \fCenter A \lif B$
  \Axiom$\Gamma_1 \fCenter A$
  \Axiom$B, \Gamma_2 \fCenter \Delta_2$
\RightLabel{$\Elim\lif$}
\TrinaryInf$\Gamma_0, \Gamma_1, \Gamma_2 \fCenter  \Delta_2$
\DisplayProof
\]

\begin{table}
  \caption{Single-conclusion natural deduction rules for some unusual connectives}
\[
\begin{array}{ccc}
\hline\hline
\text{connective} & \text{rules}\\ \hline\\ 
A \nif B & 
  \Axiom$\Gamma_1 \fCenter A$
  \Axiom$B, \Gamma_2 \fCenter$
\RightLabel{$\Intro\nif$}
\BinaryInf$\Gamma_1, \Gamma_2 \fCenter A \nif B$
\DisplayProof
\\[2ex]
\text{(exclusion)} 
& 
  \Axiom$\Gamma_0 \fCenter A \nif B$
  \Axiom$\Gamma_1 \fCenter B$
\RightLabel{$\Elim\nif_1$}
\BinaryInf$\Gamma_0, \Gamma_1 \fCenter $
\DisplayProof
& 
  \Axiom$\Gamma_0 \fCenter A \nif B$
  \Axiom$A, \Gamma_1 \fCenter D$
\RightLabel{$\Elim\nif_2$}
\BinaryInf$\Gamma_0, \Gamma_1 \fCenter D$
\DisplayProof
 \\[2ex] 
A \mid B \text{ (nand)} & 
  \Axiom$A, B, \Gamma \fCenter $
  \RightLabel{$\Intro\mid$}
  \UnaryInf$\Gamma \fCenter A \mid B$
\DisplayProof
 & 
  \Axiom$\Gamma_0 \fCenter A \mid B$
  \Axiom$\Gamma_1 \fCenter A$
  \Axiom$\Gamma_2 \fCenter B$
\RightLabel{$\Elim\mid$}
\TrinaryInf$\Gamma_1, \Gamma_2, \Gamma_3 \fCenter $
\DisplayProof
 \\[2ex]
A \nor B \text{ (nor)}
&  
    \Axiom$A, \Gamma_1 \fCenter \bot$
    \Axiom$B, \Gamma_2 \fCenter$
\RightLabel{$\Intro\nor$}
\BinaryInf$\Gamma_1, \Gamma_2 \fCenter A \nor B$
\DisplayProof 
\\[2ex] 
& 
  \Axiom$\Gamma_0 \fCenter A \nor B$
  \Axiom$\Gamma_1 \fCenter A$
\RightLabel{$\Elim\nor_1$}
\BinaryInf$\Gamma_0, \Gamma_1 \fCenter $
\DisplayProof
& 
  \Axiom$\Gamma_0 \fCenter A \nor B$
  \Axiom$\Gamma_1 \fCenter B$
\RightLabel{$\Elim\nor_2$}
\BinaryInf$\Gamma_0, \Gamma_1 \fCenter $
\DisplayProof
\\[2ex]
A \xor B & 
  \Axiom$\Gamma_1 \fCenter A$
  \Axiom$A, B, \Gamma_2 \fCenter$
\RightLabel{$\Intro\xor_1$}
\BinaryInf$\Gamma_1, \Gamma_2 \fCenter A \xor B$
\DisplayProof
 & 
  \Axiom$\Gamma_1 \fCenter B$
  \Axiom$A, B, \Gamma_2 \fCenter $
\RightLabel{$\Intro\xor_2$}
\BinaryInf$\Gamma_1, \Gamma_2 \fCenter A \xor B$
\DisplayProof 
\\[2ex]
\text{(xor)} 
& 
\multicolumn{2}{c}{
    \Axiom$\Gamma_0 \fCenter A \xor B$
    \Axiom$A, \Gamma_1 \fCenter D$
    \Axiom$B, \Gamma_2 \fCenter D$
  \RightLabel{$\Elim\xor_1$}
  \TrinaryInf$\Gamma_0, \Gamma_1, \Gamma_2 \fCenter D$
  \DisplayProof}
\\[2ex] 
& \multicolumn{2}{c}{
   \Axiom$\Gamma_0 \fCenter A \xor B$
   \Axiom$\Gamma_1 \fCenter B$
   \Axiom$\Gamma_2 \fCenter A$  \RightLabel{$\Elim\xor_2$}
  \TrinaryInf$\Gamma_0, \Gamma_1, \Gamma_2 \fCenter $
\DisplayProof}
\\[2ex]
A \lif B / C & 
    \Axiom$A, \Gamma_1 \fCenter B$
    \Axiom$\Gamma_2 \fCenter A$
  \RightLabel{$\Intro\ite_1$}
  \BinaryInf$\Gamma_1, \Gamma_2 \fCenter A \lif B/C$
\DisplayProof
& 
  \Axiom$A, \Gamma_1 \fCenter B$
  \Axiom$\Gamma_2 \fCenter C$
\RightLabel{$\Intro\ite_2$}
\BinaryInf$\Gamma_1, \Gamma_2 \fCenter A \lif B/C$
\DisplayProof
\\[2ex]
\text{(if then else)} & 
  \multicolumn{2}{c}{
      \Axiom$\Gamma_0 \fCenter A \lif B/C$
      \Axiom$A, B, \Gamma_2 \fCenter D$
      \Axiom$\Gamma_2 \fCenter A$
    \RightLabel{$\Elim\ite_1$}
    \TrinaryInf$\Gamma_0, \Gamma_1, \Gamma_2 \fCenter D$
\DisplayProof} 
\\[2ex]
& \multicolumn{2}{c}{
    \Axiom$\Gamma_0 \fCenter A \lif B/C$
    \Axiom$A, B, \Gamma_1 \fCenter D$
    \Axiom$C, \Gamma_2 \fCenter D$
  \RightLabel{$\Elim\ite_2$}
  \TrinaryInf$\Gamma_0, \Gamma_1, \Gamma_2 \fCenter D$
\DisplayProof}\\[2ex]
  \hline\hline
\end{array}
\]
\end{table}

\begin{prop}
$\Ns X$ normalizes.
\end{prop}

\begin{proof}
By inspection of the proof of Theorem~\ref{thm:normalization}. The
definition of maximal segments is now simpler; the same definition as
in \cite{Prawitz1965} applies. The cases required for the reduction of
the length of segments are now fewer in number. There is no
$\Right\Contr$ rule. The major premise of the $\Elim\opa$ rule cannot
be the conclusion of an $\Intro\opb$ rule. If it is the conclusion of
an $\opb E$ rule, the segment cannot run through the major premise.

Because the premises of the logical rules are restricted, the resolution
refutation of the clauses corresponding to the premises of the $\opa
I$ and the minor premises of the $\Elim\opa$ rule proceeds via Horn
clauses. These Horn clauses plus any context formulas result in a
derivation of the conclusion of the $\opa(\vec A)$ rule from the
premises of the of the $\Intro\opa$ and the minor premises of the $\Elim\opa$
rule.
\end{proof}

Like in the case of the sequent calculus, where we obtained a
single-conclusion system equivalent to $\L X$ by adding the $\bot_C$
rule to $\Ls X$, it is possible to obtain a classical
single-conclusion sequent-style natural deduction system.  In order to
be able to translate derivations in $\Ls X + \bot_C$ to a
single-conclusion natural deduction derivations, we have to add the
natural-deduction version of $\bot_C$ (or equivalent rules) to $\Ns
X$, such as:
\[
  \Axiom$\lnot A, \Gamma \fCenter $
\RightLabel{$\bot_C$}
\UnaryInf$\Gamma \fCenter A$
\DisplayProof
\qquad
\Axiom$\lnot A, \Gamma \fCenter C$
\Axiom$A, \Gamma \fCenter C$
\RightLabel{$\LEM$}
\BinaryInf$\Gamma, \Pi \fCenter C$
\DisplayProof
\qquad
  \Axiom$\lnot A, \Gamma \fCenter $
  \Axiom$A, \Pi \fCenter C$
\RightLabel{$\kut$}
\BinaryInf$\Gamma, \Pi \fCenter C$
\DisplayProof
\]
With $\LEM$ we can derive $\kut$, with $\kut$ we can derive $\bot_C$, and
with $\bot_C + \cut$ we can derive $\LEM$.

The proof of normalization for classical single-conclusion natural
deduction including the rule $\bot_C$ is more complicated than both the
pure intuitionist case and the multi-conclusion case. This is due
to the fact that the notion of maximal segment must be extended to include
sequences of formulas beginning with the conclusion of $\bot_C$ and
ending with a major premise of an elimination rule. Such segments are
harder to remove. Although \cite{Prawitz1965} already showed
normalization for the fragment excluding $\lor$, normalization of the
full system was not established until \cite{Stalmarck1991}.

\begin{conj}
  $\Ns X + \bot_C$ normalizes.
\end{conj}

\section{Proof terms and formulas-as-types}
\label{sec:ch}

Under the Curry-Howard isomorphism, proofs in $\Nsl X$ can be assigned
\emph{proof terms}, just like they can in the standard intuitionistic
cases.  Under this isomorphism, conversions of inference segments used
to reduce or permute inferences correspond to operations on proof
terms.  The general case considered here suggests that it might be
fruitful to separate the proof operation of substitution, which
corresponds to the cut rule, from the conversions themselves.

For each connective $\opa$, we may introduce two function symbols, one
that applies to premises of an introduction rule (a proof
constructor), and one to the premises of an elimination rule (a proof
destructor).  The types of these functions are given by the
corresponding auxiliary formulas; and their conversion clauses by the
corresponding derivation simplifications obtained from resolution
refutations.  Labels of discharged assumptions are abstracted similar
to $\lambda$-abstraction. We begin by considering some familiar examples.

Conjunction can be given a single introduction rule in Horn form,
\[
  \Axiom$\Gamma_1 \fCenter A$
  \Axiom$\Gamma_2 \fCenter B$
\RightLabel{$\Intro\land$}
\BinaryInf$\Gamma_1, \Gamma_2 \fCenter A \land B$
\DisplayProof
\]
The constructor corresponding to $\Intro\land$ is the pairing function,
$c^\land$. It takes two arguments, one for each premise. If the
premises end in derivations labelled by $s$ and $t$, respectively, we
label the derivation ending in the conclusion by $c^\land(s, t)$. We
may include this in the statement of the rule itself, by adding a type
annotation on the right:
\[
\Axiom$\Gamma_1 \fCenter s \colon A$
\Axiom$\Gamma_2 \fCenter t \colon B$
\RightLabel{$\Intro\land$}
\BinaryInf$\Gamma_1, \Gamma_2 \fCenter c^\land(s, t) \colon A \land
  B$
\DisplayProof
\]
Split general elimination rules are
\[
  \Axiom$\Gamma_0 \fCenter A \land B$
  \Axiom$A, \Gamma_1 \fCenter \Lambda$
\RightLabel{$\Elim\land_1$}
\BinaryInf$\Gamma_0, \Gamma_1 \fCenter \Lambda$
\DisplayProof
\qquad
  \Axiom$\Gamma_0 \fCenter A \land B$
  \Axiom$B, \Gamma_1 \fCenter \Lambda$
\RightLabel{$\Elim\land_2$}
\BinaryInf$\Gamma_0, \Gamma_1 \fCenter \Lambda$
\DisplayProof
\]
They correspond to the two destructors $d_1^\land$, $d_2^\land$, and
the rules may be written:
\[
  \Axiom$\Gamma_0 \fCenter t \colon A \land B$
  \Axiom$x \colon A, \Gamma_1 \fCenter s_1 \colon C$
\RightLabel{$\Elim\land_1$}
\BinaryInf$\Gamma_0, \Gamma_1 \fCenter d_1^\land(t, [x]s_1) \colon C$
\DisplayProof
\qquad
  \Axiom$\Gamma_0 \fCenter t \colon A \land B$
  \Axiom$x \colon B, \Gamma_1 \fCenter s_2 \colon C$
\RightLabel{$\Elim\land_2$}
\BinaryInf$\Gamma_0, \Gamma_1 \fCenter d_2^\land(t, [x]s_2) \colon C$
\DisplayProof
\]
Such a destructor may be seen as a generalized $\lambda$ abstraction
operator; the $x$ in the second argument is bound.

A \cut{} inference is treated using a substitution operator,
\[
  \Axiom$\Gamma \fCenter s \colon A$
  \Axiom$x \colon A, \Pi \fCenter t \colon C$ 
\RightLabel{$\cut \colon x$}
\BinaryInf$\Gamma, \Pi \fCenter \subst(s, x, [x]t) \colon C$
\DisplayProof
\]
It too binds the label of the discharged assumption $x \colon A$.

An $\Intro\land$ inference followed by an $\Elim\land_1$ inference may be
reduced from
\[
\bottomAlignProof
    \Axiom$\Gamma_1 \fCenter s \colon A$ 
    \Axiom$\Gamma_2 \fCenter t \colon B$
  \RightLabel{$\Intro\land$}
  \BinaryInf$\Gamma_1, \Gamma_2 \fCenter c^\land(s, t) \colon A \land B$
  \Axiom$x \colon A, \Gamma_3 \fCenter u \colon C$
\RightLabel{$\Elim\land_1$}
\BinaryInf$\Gamma_1, \Gamma_2, \Gamma_3 \fCenter 
  d^\land_1(c^\land(s, t), [x]u)$
\DisplayProof
\]
to
\[
\bottomAlignProof
\Axiom$\Gamma_1 \fCenter s \colon A$ 
\Axiom$x \colon A, \Gamma_3 \fCenter u \colon C$
\RightLabel{$\cut \colon x$}
\BinaryInf$\Gamma_1, \Gamma_3 \fCenter 
  \subst(s, x, [x]u) \colon C$
\DisplayProof
\]
and similarly for an $\Intro\land$ followed by a $\Elim\land_2$ inference.

The corresponding reductions for proof terms then are:
\begin{align*}
d^\land_1(c^\land(s, t), [x]u) & \rightarrow \subst(s, x, [x]u) \\
d^\land_2(c^\land(s, t), [x]u) & \rightarrow \subst(t, x, [x]u) \\
\end{align*}

A single general elimination rule for $\land$ is given by
\[
  \Axiom$\Gamma_0 \fCenter A \land B$
  \Axiom$A, B, \Gamma_1 \fCenter C$
\RightLabel{$\Elim\land$}
\BinaryInf$\Gamma_0, \Gamma_1 \fCenter C$
\DisplayProof
\]
The destructor corresponding to it is $d^\land$, it takes two
arguments, the proof terms for the major and minor premise.  The
labels of the discharged assumptions are abstracted.  The rule with
proof terms is as follows:
\[
  \Axiom$\Gamma_0 \fCenter s \colon A \land B$
  \Axiom$x \colon A, y \colon B, \Gamma_1 \fCenter t \colon C$
\RightLabel{$\Elim\land$}
\BinaryInf$\Gamma_0, \Gamma_1 \fCenter d^\land(s, [x, y]t) \colon C$
\DisplayProof
\]
The proof simplification corresponding to a $\Intro\land$ followed by $\Elim\land$ is:
\[
    \Axiom$\Gamma_1 \fCenter s \colon A$ 
    \Axiom$\Gamma_2 \fCenter t \colon B$
  \RightLabel{$\Intro\land$}
  \BinaryInf$\Gamma_1, \Gamma_2 \fCenter c^\land(s, t) \colon A \land B$
  \Axiom$x \colon A, y \colon B, \Gamma_3 \fCenter u \colon C$
\RightLabel{$\Elim\land$}
\BinaryInf$\Gamma_1, \Gamma_2, \Gamma_3 \fCenter 
  d^\land(c^\land(s, t), [x,y]u)$
\DisplayProof
\]
to
\[
  \Axiom$\Gamma_2 \fCenter t \colon B$
    \Axiom$\Gamma_1 \fCenter s \colon A$
    \Axiom$x \colon A, y \colon B, \Gamma_3 \fCenter u \colon C$
  \RightLabel{$\cut \colon x$}
  \BinaryInf$y \colon B,\Gamma_1, \Gamma_3 \fCenter \subst(s, x,
    [x]u) \colon C$
\RightLabel{$\cut \colon y$}
\BinaryInf$\Gamma_1, \Gamma_2, \Gamma_3 \fCenter \subst(t,
  y, (\subst(s, x, [x]u)) \colon C$
\DisplayProof
\]
Correspondingly, the $\beta$-conversion rule for $c^\land/d^\land$ is
\[
d^\land(c^\land(s, t), [x,y]u) \rightarrow 
\subst(t, y, \subst(s, x, [x,y]u)) 
\] 

For the conditional, the introduction and general elimination rules are
\[
  \Axiom$x \colon A, \Gamma \fCenter s \colon B$
\RightLabel{$\Intro\lif$}
\UnaryInf$\Gamma \fCenter c^\lif([x]s) \colon A \lif B$
\DisplayProof
\]\[
\Axiom$\Gamma_1 \fCenter s \colon A \lif B$
\Axiom$\Gamma_2 \fCenter t \colon A$
\Axiom$x \colon B, \Gamma_3 \fCenter u \colon C$
\RightLabel{$\Elim\lif$}
\TrinaryInf$\Gamma_1, \Gamma_2, \Gamma_3 \fCenter d^\lif(s, t, [x]u) \colon C$
\DisplayProof
\]
Here $c^\lif([x]s)$ is just a complex way of writing $\lambda x.s$.
$d^\lif(s, t, [x]u)$ is the generalized application operator of
\cite{JoachimskiMatthes2003}.  The corresponding proof simplification
obtained from a resolution refutation of the premises is
\[
    \Axiom$x \colon A, \Gamma_1 \fCenter s \colon B$
  \RightLabel{$\Intro\lif$}
  \UnaryInf$\Gamma_1 \fCenter c^\lif([x]s) \colon A \lif B$
  \Axiom$\Gamma_2 \fCenter t \colon A$
  \Axiom$x \colon B, \Gamma_3 \fCenter u \colon C$
\RightLabel{$\Elim\lif$}
\TrinaryInf$\Gamma_1, \Gamma_2, \Gamma_3 \fCenter 
  d^\lif(c^\lif([x]s), t, [x]u) \colon C$
\DisplayProof
\]
to
\[
    \Axiom$\Gamma_2 \fCenter t \colon A$
    \Axiom$x \colon A, \Gamma_1 \fCenter s \colon B$
  \RightLabel{$\cut \colon x$}
  \BinaryInf$\Gamma_1, \Gamma_2 \fCenter \subst(t, x, [x]s) \colon B$
  \Axiom$x \colon B, \Gamma_3 \fCenter u \colon C$
\RightLabel{$\cut \colon x$}
\BinaryInf$\Gamma_1, \Gamma_2, \Gamma_3 \fCenter 
\subst(\subst(t, x, [x]s), x, [x]u) \colon C$
\DisplayProof
\]
The $\beta$-reduction rule for the $c^\lif/d^\lif$ pair is then:
\[
d^\lif(c^\lif([x]s), t, [x]u) \to \subst(\subst(t, x, [x]s), x, [x]u)
\]
or, if we prefer standard notation
\[
\mathrm{app}(\lambda x.s, t, [x]u) \to u[s[t/x]/x]
\]

In general, we have the following situation: For a connective with
introduction rules $\Intro\opa_k$ and elimination rules $\Elim\opa_l$ we
introduce function symbols: constructors $c_k^\opa$ and destructors
$d^\opa_l$.  Each such function symbol has the same arity as the
corresponding rule has premises.  In an application of the rule, the
succedent of the conclusion is labelled by the term
\[
c^\opa_k([\vec x_1^k]s_1^k, \dots, [\vec x_n^k]s_n^k) \text{ or }
d^\opa_l(s_0, [\vec x_1^l]s_1^l, \dots, [\vec x_m^l]s_m^l),
\]
where the $s_i$ are the terms labelling the succedents of the premise
sequents and $\vec x_i$ are the labels of the assumptions discharged
in the $i$-th premise.

The $\beta$-reduction rule for a redex of the form
\[
d^\opa_l(c^\opa_k([\vec x_1^k]s_1^k, \dots, [\vec x_n^k]s_n^k), [\vec
  x_1^l]s_1^l, \dots, [\vec x_m^l]s_m^l)
\]
is provided by the corresponding simplification conversion on proofs,
which in turn is given by a derivation segment consisting only of cut
rules, i.e., a resolution refutation of the premises of the $\Intro\opa_i$
and $\Elim\opa_j$ rules.  This simplification conversion produces a term
built from the arguments $s_i^k$ and $s_j^l$ using only $\subst$
operations. There are corresponding general $\alpha$ conversions:
replace any subterm of the form $[x]s$ by $[y]s'$ where $s'$ is the
result of substituting $y$ for every free occurrence of~$x$ in $s$,
provided $y$ is free for $x$ in~$s$. $\eta$-conversion corresponds to
converting between $[x]s$ and $s$, provided $x$ is not free in~$s$.
Finally, we define substitution redexes $\subst(t, y, [\vec x]s)$ which
convert to $[\vec z]s'$, where $s'$ is $s$ with every free occurrence
of $y$ in $s$ replaced by $t$, provided $t$ is free for $y$ in~$s$ and
$y$ occurs in~$\vec x$, and $\vec z$ is $\vec x$ without~$y$. If $y
\notin \vec x$, then $\subst(t, y, [\vec x]s)$ converts to~$[\vec x]s$.

\begin{conj}
The typed $\lambda$-calculus obtained in this way strongly normalizes.
\end{conj}

\section{Free Deduction}
\label{sec:fd}

\cite{Parigot1992} has introduced a calculus he called \emph{free
deduction}~$\FD$. It has pairs of rules for each connective called
left and right elimination rules. The right elimination rules are just
the general elimination rules of sequent-style natural deduction. Left
elimination rules are what \cite{Milne2015} has called ``general
introduction rules.''  Rather than adding the complex formula
containing a connective as a conclusion, they allow assumptions
containing the complex formula to be discharged. They are thus
perfectly symmetric with the general elimination rule. The major
premise contains the principal formula not on the right (as a
conclusion) but on the left (as a discharged assumption).

For instance, the $\FD$ rules for $\lif$ are:
\[
  \Axiom$A \lif B, \Gamma \fCenter \Delta$
  \Axiom$A, \Pi \fCenter \Sigma$ 
\RightLabel{$\LElim\lif_1$}
\BinaryInf$\Gamma, \Pi \fCenter \Delta, \Sigma$
\DisplayProof
\qquad
  \Axiom$A \lif B, \Gamma \fCenter \Delta$
  \Axiom$\Pi \fCenter \Sigma, B$
\RightLabel{$\LElim\lif_2$}
\BinaryInf$\Gamma, \Pi \fCenter \Delta, \Sigma$
\DisplayProof
\]
and
\[
    \Axiom$\Gamma \fCenter \Delta, A \lif B$
    \Axiom$\Pi \fCenter \Sigma, A$
    \Axiom$B, \Pi' \fCenter \Sigma$
\RightLabel{$\RElim\lif$}
\TrinaryInf$\Gamma, \Pi, \Pi' \fCenter \Delta, \Sigma, \Sigma'$
\DisplayProof
\]
This is a multi-conclusion version of Milne's natural deduction system
with general elimination (right elimination) and general introduction
(left elimination) rules. Here the $\LElim\lif$ rule is split; an
equivalent version has a single rule,
\[
  \Axiom$A \lif B, \Gamma \fCenter \Delta$
  \Axiom$A, \Pi \fCenter \Sigma, B$
\RightLabel{$\LElim\lif$}
\BinaryInf$\Gamma, \Pi \fCenter \Delta, \Sigma$
\DisplayProof
\]
Parigot presents this (and non-split rules for $\land$, $\lor$) as a
variant system~$\FD'$.

Free deduction rules (and hence Milne's ``general introduction
rules'') can be formulated for arbitrary connectives $\opa$ along the
same lines as for sequent calculus and natural deduction.  Given an
introduction rule for $\opa$ with premises $\Pi_i, \Gamma \fCenter
\Delta, \Lambda_i$, the corresponding $\LElim\opa$ rule is
\[
  \Axiom$\opa(\vec A), \Gamma \fCenter \Delta$ 
  \Axiom$\Pi_1, \Gamma \fCenter \Delta, \Lambda_1$
  \AxiomC{\dots}
  \Axiom$\Pi_n, \Gamma \fCenter \Delta, \Lambda_n$
\RightLabel{$\LElim\opa$}
\QuaternaryInf$\Gamma \fCenter \Delta$
\DisplayProof
\]

Free deduction embeds both sequent calculus and natural deduction.
The left and right sequent calculus rules can be simulated by taking
the major premise in the corresponding free deduction right or left
elimination rule to be an initial sequent.  The introduction rule of
natural deduction is obtained the same way as the right rule of the
sequent calculus from the $\LElim{}$ rule, e.g.,
\[
  \Axiom$\opa(\vec A) \fCenter \opa(\vec A)$ 
  \Axiom$\Pi_1, \Gamma \fCenter \Delta, \Lambda_1$
  \AxiomC{\dots}
  \Axiom$\Pi_n, \Gamma \fCenter \Delta, \Lambda_n$
\RightLabel{$\LElim\opa$}
\QuaternaryInf$\Gamma \fCenter \Delta, \opa(\vec A)$
\DisplayProof
\]
Cuts in this system are just like segments in natural deduction,
except that here the formula $\opa(\vec A)$ which appears in the major
premise of an $\Elim\opa$ rule at the end of the segment is---not the
conclusion of an $\Intro\opa$ rule as in natural deduction,
but---discharged by an application of $\LElim\opa$ at the beginning of a
segment.  A segment of length~1 then is of the form,
\[
  \Axiom$\opa(\vec A), \Gamma \fCenter \Delta, \opa(\vec A)$
  \Axiom$\cdots \Pi_i, \Gamma \fCenter \Delta, \Lambda_i \cdots$
  \RightLabel{$\LElim\opa$}
  \BinaryInf$\Gamma \fCenter \Delta, \opa(\vec A)$
  \Axiom$\cdots\Pi_j', \Gamma \fCenter \Delta, \Lambda_j'\cdots$
\RightLabel{$\RElim\opa$}
\BinaryInf$\Gamma \fCenter \Delta$
\DisplayProof
\]
The cut-elimination mechanism of $\FD$ can be straightforwardly
adapted to the generalized case.  In this situation, we have to
introduce labels not just for assumptions (formulas occurring on the
left of a sequent) but also for those on the right.  A maximal segment
can be replaced by a sequence of cuts. As in he case of natural
deduction, cuts correspond to proof substitutions. In $\FD$, however,
not only can we substitute derivations of $\Gamma \fCenter A$ for
assumptions in derivations of $x \colon A, \Pi \fCenter \Lambda$, but
also conversely substitute a derivation of $A, \Pi \fCenter \Lambda$ in
a derivation of $\Gamma \fCenter \Lambda, x \colon A$.  This of course
results in a nondeterministic reduction procedure, as a cut inference
may be replaced either by a substitution of the conclusion of the left
premise in corresponding assumptions in the proof ending in the right
premise, or the other way around.

\section{Quantifiers}
\label{sec:quantifiers}

The basic principle underlying the generation of sequent calculus and
natural deduction rules with the usual proof-theoretic properties can
be extended to quantifiers as well.  The important additional aspect
of quantifier rules is that they (sometimes) require
\emph{eigenvariable} conditions, as in the case of $\Right\forall$.
Eigenvariable conditions are necessary for soundness, but they also
guarantee that terms can be substituted in a proof for a variable
appearing in an auxiliary formula, which is crucial for the
cut-elimination and normalization properties.

The correspondence between a sequent calculus rule and a set of
clauses also holds for quantifiers.  The most general case of a
quantifier for which this model could be considered has a fixed finite
number of bound variables and a fixed finite number of schematic
subformulas which may contain these variables: \[\Q x_1 \dots
x_m(A_1(x_1, \dots, x_m), \dots, A_n(x_1, \dots, x_m)).\] The truth
conditions of such a quantifier $\Q$ may be given by a set~$\C$ of
clauses, i.e., a conjunction of disjunctions of literals, where each
atom is of the form $A_i(t_1, \dots, t_m)$, and where $t_j$ is either a
variable $x_j$ or a Skolem term~$f(x_1, \dots, x_k)$ with $k < j$.
Under this interpretation, $\Q \vec x(A_1(\vec x), \dots, A_n(\vec
x))$ is true iff $\exists \vec f\forall \vec x \C$. Similarly, falsity
conditions may be given for $\Q$ as well.

Such a set of clauses $\C$ corresponds to a sequent calculus rule for
$\Right\Q$, by interpreting a variable $x_j$ as an eigenvariable~$a_j$ and
a Skolem term $f(x_1, \dots, x_k)$ as schematic variable for a term
which may contain the eigenvariables $a_1$, \dots, $a_k$). If the
truth and falsity conditions for $\Q$ underlying the $\Left\Q$ and $\Right\Q$
rules are mutually exclusive (i.e., in each interpretation at most one
of the two is satisfied), the resulting sequent calculus rules are
sound. However, they will be complete only if it is not the case that
the left and right rules each contain both term and eigenvariables.

Let us consider some examples.  The syllogistic quantifiers, e.g.,
``All $A$s are $B$s'' and ``Some $A$s are $B$s,'' fit this scheme:
\begin{align*}
\textsf{A} x(A(x), B(x)) & \text{ iff } \forall x(\lnot A(x) \lor B(x)) \\
\lnot \textsf{A} x(A(x), A(x)) & \text{ iff } \exists f(A(f) \land \lnot B(f)) \\
\textsf{S} x(A(x), B(x)) & \text{ iff } \exists f(A(f) \land B(f)) \\
\lnot \textsf{S} x(A(x), B(x)) & \text{ iff } \forall x(\lnot A(x) \lor \lnot B(x))
\end{align*}
We obtain the rules
\[
  \Axiom$\Gamma \fCenter \Delta, A(t)$
  \Axiom$B(t), \Gamma \fCenter \Delta$
\RightLabel{$\Left{\textsf{A}}$}
\BinaryInf$\textsf{A} x(A(x),B(x)), \Gamma \fCenter \Delta$
\DisplayProof
\qquad
\Axiom$A(a), \Gamma \fCenter \Delta, B(a)$
\RightLabel{$\Right{\textsf{A}}$}
\UnaryInf$\Gamma \fCenter \Delta, \textsf{A}x(A(x), B(x))$
\DisplayProof
\]
and
\[
\Axiom$A(a), B(a), \Gamma \fCenter \Delta$
\RightLabel{$\Left{\textsf{S}}$}
\UnaryInf$\textsf{S}x(A(x), B(x)), \Gamma \fCenter \Delta$
\DisplayProof
\qquad
\Axiom$\Gamma \fCenter \Delta, A(t)$
\Axiom$\Gamma \fCenter \Delta, B(t)$
\RightLabel{$\Right{\textsf{S}}$}
\BinaryInf$\Gamma \fCenter \Delta, \textsf{S}x(A(x), B(x))$
\DisplayProof
\]
Sch\"onfinkel's generalized Sheffer stroke $A(x) \mid^x B(x)$ is the
dual of $\textsf{S}$:
\begin{align*}
\textsf{U} x(A(x), B(x)) & \text{ iff } \forall x(\lnot A(x) \lor \lnot B(x)) \\
\lnot \textsf{U} x(A(x), B(x)) & \text{ iff } \exists f(A(f) \land B(f))
\end{align*}
and has the rules
\[
  \Axiom$\Gamma \fCenter \Delta, A(t)$
  \Axiom$\Gamma \fCenter \Delta, B(t)$
\RightLabel{$\Left{\textsf{U}}$}
\BinaryInf$\textsf{U}x(A(x), B(x)), \Gamma \fCenter \Delta$
\DisplayProof
\qquad
\Axiom$A(a), B(a), \Gamma \fCenter \Delta$
\RightLabel{$\Right{\textsf{U}}$}
\UnaryInf$\Gamma \fCenter \Delta, \textsf{U}x(A(x), B(x))$
\DisplayProof
\]

We can also consider non-monadic quantifiers, e.g., 
the totality quantifier $\textsf{T}xy\,A(x, y)$, expressed by
\begin{align*}
\textsf{T} xy\,A(x, y) & \text{ iff } \exists f\forall x\,A(x, f(x)) \\
\lnot  \textsf{T} xy\,A(x, y) & \text{ iff } \exists g\forall x\,\lnot A(g, x)
\end{align*}
with the rules
\[
  \Axiom$A(s, b), \Gamma \fCenter \Delta$
\RightLabel{$\Left{\textsf{T}}$}
\UnaryInf$\textsf{T}xy\,A(x, y), \Gamma \fCenter \Delta$
\DisplayProof
\qquad
\Axiom$\Gamma \fCenter \Delta, A(a, t(a))$
\RightLabel{$\Right{\textsf{T}}$}
\UnaryInf$\Gamma \fCenter \Delta, \textsf{T}xy\,A(x, y)$
\DisplayProof
\]
In $\Left{\textsf{T}}$, the eigenvariable $a$ may not appear in the
conclusion, but also not in~$t$, while in $\Right{\textsf{T}}$ the
eigenvariable $a$ may not appear in the conclusion, but is allowed to
appear in~$t(a)$. These rules are sound, but not complete. For
instance, there is no derivation of $\textsf{T}xy\, A(x, y) \fCenter
\textsf{T}xy\, A(x,y)$ from instances of $A(x, y) \fCenter A(x, y)$, since
the eigenvariable condition is violated if $\Left{\textsf{T}}$ or
$\Right{\textsf{T}}$ is applied to such an initial sequent.

Of course, not all natural quantifiers can even be provided with rules
using this framework.  For instance, the Henkin quantifier
$\left\{{\forall x\exists y}\atop{\forall u\exists v}\right\} A(x, y,
u, v)$ has truth conditions given by
\[
\exists f\exists g\forall x\forall u\,A(x,f(x),u,g(u))
\]
but its falsity conditions cannot be stated in this
form.

Whether or not the rules obtained this way are complete, they always
enjoy cut-elimination. Since the clause sets $\C$ and $\C'$ corresponding
to the $\Left\Q$ a $\Right\Q$ rules are mutually exclusive, $\C \cup
\C'$ is an unsatisfiable set of clauses.  Thus a \mix{} inference in
which the left premise is the conclusion of $\Right\Q$ and the right
premise the conclusion of $\Left\Q$ can be reduced to \mix{} inferences
operating on the premises of the $\Left\Q$ and $\Right\Q$ rules, or on
sequents obtained from them by substituting terms for eigenvariables.
For instance, consider
\[
    \Axiom$\Gamma \fCenter \Delta, A(a, t(a))$
  \RightLabel{$\Right{\textsf{T}}$}
  \UnaryInf$\Gamma \fCenter \Delta, \textsf{T}xy\,A(x, y)$
    \Axiom$A(s, b), \Pi \fCenter \Lambda$
  \RightLabel{$\Left{\textsf{T}}$}
  \UnaryInf$\textsf{T}xy\,A(x, y), \Pi \fCenter \Lambda$
\RightLabel{$\cut$}
\BinaryInf$\Gamma, \Pi \fCenter \Delta, \Lambda$
\DisplayProof
\]
The clause set $\{\{\lnot A(g, x)\}, \{A(y, f(y))\}\}$ is refutable by
a single resolution inference, with most general unifier $\{g \mapsto
y, f(g) \mapsto x\}$. This means the eigenvariable~$a$ corresponding
to $y$ will be substituted by the term~$s$ corresponding to~$g$, and
the eigenvariable~$b$ corresponding to $x$ will be substituted by the
term corresponding to~$f(g)$, i.e., $t(s)$.  The restriction that $b$
must not occur in $s$ guarantees that this latter substitution
performed in the proof ending in the premise of $\textsf{T}R$ can be
carried out.  We can then replace the cut by
\[
  \Axiom$\Pi \fCenter \Lambda, A(s, t(s))$
  \Axiom$A(s, t(s)), \Gamma \fCenter \Delta$
\RightLabel{$\cut$}
\BinaryInf$\Gamma, \Pi \fCenter \Delta, \Lambda$
\DisplayProof
\]

We obtain natural deduction rules just as in the propositional case;
normalization holds here as well for the same reason: the clause sets
corresponding to the premises of a $\Intro\Q$ rule together with the minor
premises of a $\Elim\Q$ rule are jointly unsatisfiable, hence refutable
by resolution. The resolution refutation translates into a
sequence of cut inferences applied to derivations ending in these
premises, possibly with eigenvariables replaced by terms.  The
substitution for each cut inference is provided by the unifier in the
corresponding resolution inference; eigenvariable conditions guarantee
that the corresponding substitution of terms for eigenvariables can be
applied to the entire derivation. (This is the idea behind the ``cut
elimination by resolution'' method of \cite{BaazLeitsch2000}.)

\subsection*{Acknowledgements}

The results in this paper were presented at the 2016 Annual Meeting of
the Association for Symbolic Logic 2016 (\cite{Zach2017a}) and at the
Ohio State University/UConn Workshop on Truth in 2017. The author
  would like to thank audiences there and the referees for the \emph{Review}
for their helpful comments and suggestions.

\bibliographystyle{asl}
\bibliography{gencalcs}

@inproceedings{BaazFermuller1996,
  address = {{Los Alamitos}},
  title = {Intuitionistic Counterparts of Finitely-Valued Logics},
  isbn = {978-0-8186-7392-4},
  language = {en},
  booktitle = {26th {{International Symposium}} on {{Multiple}}-Valued {{Logic}}. {{Proceedings}}},
  publisher = {{IEEE Press}},
  doi = {10.1109/ISMVL.1996.508349},
  author = {Baaz, Matthias and Ferm{\"u}ller, Christian G.},
  year = {1996},
  pages = {136-141}
}

@article{Cellucci1992,
  title = {Existential Instantiation and Normalization in Sequent Natural Deduction},
  volume = {58},
  issn = {01680072},
  language = {en},
  number = {2},
  journal = {Annals of Pure and Applied Logic},
  doi = {10.1016/0168-0072(92)90002-H},
  author = {Cellucci, Carlo},
  month = sep,
  year = {1992},
  pages = {111-148}
}

@book{NegrivonPlato2001,
  address = {{Cambridge}},
  title = {Structural Proof Theory},
  isbn = {978-0-521-79307-0},
  lccn = {QA9.54 .N44 2001},
  language = {en},
  publisher = {{Cambridge University Press}},
  author = {Negri, Sara and {von Plato}, Jan},
  year = {2001}
}

@inproceedings{Parigot1992a,
  series = {Lecture {{Notes}} in {{Computer Science}}},
  title = {$\lambda\mu$-Calculus: {{An}} Algorithmic Interpretation of Classical Natural Deduction},
  volume = {624},
  shorttitle = {{$\Lambda\mu$}-{{Calculus}}},
  language = {en},
  booktitle = {Logic {{Programming}} and {{Automated Reasoning}}},
  publisher = {{Springer, Berlin, Heidelberg}},
  doi = {10.1007/BFb0013061},
  author = {Parigot, Michel},
  year = {1992},
  pages = {190-201}
}

@incollection{Milne2015,
  address = {{Berlin}},
  series = {Outstanding {{Contributions}} to {{Logic}}},
  title = {Inversion Principles and Introduction Rules},
  volume = {7},
  isbn = {978-3-319-11040-0},
  language = {en},
  booktitle = {Dag {{Prawitz}} on {{Proofs}} and {{Meaning}}},
  publisher = {{Springer}},
  doi = {10.1007/978-3-319-11041-7_8},
  author = {Milne, Peter},
  editor = {Wansing, Heinrich},
  year = {2015},
  pages = {189-224}
}

@article{JoachimskiMatthes2003,
  title = {Short Proofs of Normalization for the Simply-Typed {$\lambda$}-Calculus, Permutative Conversions and {{G{\"o}del}}'s {{T}}},
  volume = {42},
  issn = {0933-5846, 1432-0665},
  language = {en},
  number = {1},
  journal = {Archive for Mathematical Logic},
  doi = {10.1007/s00153-002-0156-9},
  author = {Joachimski, Felix and Matthes, Ralph},
  month = jan,
  year = {2003},
  pages = {59-87}
}

@article{Stalmarck1991,
  title = {Normalization Theorems for Full First Order Classical Natural Deduction},
  volume = {56},
  issn = {0022-4812, 1943-5886},
  language = {en},
  number = {1},
  journal = {The Journal of Symbolic Logic},
  doi = {10.2307/2274910},
  author = {St{\aa}lmarck, Gunnar},
  month = mar,
  year = {1991},
  pages = {129-149}
}

@phdthesis{Baaz1984,
  type = {Dissertation},
  title = {Die {{Anwendung}} Der {{Methode}} Der {{Sequenzialkalk{\"u}le}} Auf Nichtklassische {{Logiken}}},
  school = {Universit{\"a}t Wien},
  author = {Baaz, Matthias},
  year = {1984}
}

@article{BaazFermullerZach1994,
  title = {Elimination of Cuts in First-Order Finite-Valued Logics},
  volume = {29},
  copyright = {All rights reserved},
  number = {6},
  journal = {Journal of Information Processing and Cybernetics EIK},
  author = {Baaz, Matthias and Ferm{\"u}ller, Christian G. and Zach, Richard},
  year = {1994},
  pages = {333-355},
  eprint = {BAAEOC},
  eprinttype = {philpapers},
  scholar = {15462298465941074350}
}

@inproceedings{BaazFermullerZach1993b,
  address = {{Los Alamitos}},
  title = {Systematic Construction of Natural Deduction Systems for Many-Valued Logics},
  copyright = {All rights reserved},
  booktitle = {23rd {{International Symposium}} on {{Multiple}}-Valued {{Logic}}. {{Proceedings}}},
  publisher = {{IEEE Press}},
  doi = {10.1109/ISMVL.1993.289558},
  author = {Baaz, Matthias and Ferm{\"u}ller, Christian G. and Zach, Richard},
  year = {1993},
  pages = {208-213},
  eprint = {FERSCO-3},
  eprinttype = {philpapers},
  scholar = {8129704214380989574,8943239453694643393}
}

@article{Schroeder-Heister1984,
  title = {A Natural Extension of Natural Deduction},
  volume = {49},
  issn = {1943-5886},
  number = {4},
  journal = {Journal of Symbolic Logic},
  doi = {10.2307/2274279},
  author = {{Schroeder-Heister}, Peter},
  month = dec,
  year = {1984},
  pages = {1284-1300},
  numpages = {17}
}

@article{Zach2016,
  title = {Natural Deduction for the {{Sheffer}} Stroke and {{Peirce}}'s Arrow (and Any Other Truth-Functional Connective)},
  volume = {45},
  copyright = {All rights reserved},
  issn = {0022-3611, 1573-0433},
  language = {en},
  number = {2},
  journal = {Journal of Philosophical Logic},
  doi = {10.1007/s10992-015-9370-x},
  author = {Zach, Richard},
  month = apr,
  year = {2016},
  pages = {183-197},
  eprint = {ZACNDF},
  eprinttype = {philpapers},
  scholar = {7700562534442313606}
}

@phdthesis{Zach1993,
  address = {{Vienna, Austria}},
  type = {Diplomarbeit},
  title = {Proof Theory of Finite-Valued Logics},
  copyright = {All rights reserved},
  school = {Technische Universit{\"a}t Wien},
  author = {Zach, Richard},
  year = {1993},
  scholar = {9133519663533066510,10148053650071006680}
}

@book{Prawitz1965,
  address = {{Stockholm}},
  series = {Stockholm {{Studies}} in {{Philosophy}}},
  title = {Natural Deduction: A Proof-Theoretical Study},
  number = {3},
  publisher = {{Almqvist \& Wiksell}},
  author = {Prawitz, Dag},
  year = {1965},
  note = {Reprinted by Dover, 2006}
}

@article{Gentzen1934,
  title = {{Untersuchungen {\"u}ber das logische Schlie{\ss}en I\textendash{}II}},
  volume = {39},
  language = {de},
  number = {1},
  journal = {Mathe\-mati\-sche Zeit\-schrift},
  doi = {10.1007/BF01201353},
  author = {Gentzen, Gerhard},
  year = {1934},
  pages = {176-210, 405-431},
  note = {English translation in \cite[pp.~68--131]{Gentzen1969}}
}

@inproceedings{Parigot1992,
  address = {{Berlin, Heidelberg}},
  series = {Lecture {{Notes}} in {{Computer Science}}},
  title = {Free Deduction: {{An}} Analysis of ``Computations'' in Classical Logic},
  volume = {592},
  isbn = {978-3-540-55460-8},
  language = {en},
  booktitle = {Logic {{Programming}}},
  publisher = {{Springer Berlin Heidelberg}},
  doi = {10.1007/3-540-55460-2_27},
  author = {Parigot, Michel},
  editor = {Voronkov, Andrej},
  year = {1992},
  pages = {361-380}
}

@incollection{Kutschera1962,
  address = {{Freiburg and Munich}},
  title = {{Zum Deduktionsbegriff der klassischen Pr{\"a}dikatenlogik erster Stufe}},
  language = {de},
  booktitle = {{Logik und Logikkalk{\"u}l}},
  publisher = {{Karl Alber}},
  author = {{von Kutschera}, Franz},
  editor = {K{\"a}sbauer, Max and {von Kutschera}, Franz},
  year = {1962},
  pages = {211-236}
}

@article{Boricic1985,
  title = {On Sequence-Conclusion Natural Deduction Systems},
  volume = {14},
  issn = {1573-0433},
  language = {en},
  number = {4},
  journal = {Journal of Philosophical Logic},
  doi = {10.1007/BF00649481},
  author = {Bori{\v c}i{\'c}, Branislav R.},
  month = nov,
  year = {1985},
  pages = {359-377}
}

@article{SambinBattilottiFaggian2000,
  title = {Basic Logic: Reflection, Symmetry, Visibility},
  volume = {65},
  issn = {0022-4812, 1943-5886},
  shorttitle = {Basic Logic},
  language = {en},
  number = {3},
  journal = {The Journal of Symbolic Logic},
  doi = {10.2307/2586685},
  author = {Sambin, Giovanni and Battilotti, Giulia and Faggian, Claudia},
  month = sep,
  year = {2000},
  pages = {979-1013}
}

@article{Schroeder-Heister2013,
  series = {Articles in Honor of {{Giovanni Sambin}}'s 60th Birthday},
  title = {Definitional Reflection and Basic Logic},
  volume = {164},
  issn = {0168-0072},
  number = {4},
  journal = {Annals of Pure and Applied Logic},
  doi = {10.1016/j.apal.2012.10.010},
  author = {{Schroeder-Heister}, Peter},
  month = apr,
  year = {2013},
  pages = {491-501}
}

@book{ShoesmithSmiley1978,
  address = {{Cambridge}},
  title = {Multiple-Conclusion Logic},
  publisher = {{Cambridge University Press}},
  author = {Shoesmith, David J. and Smiley, Timothy J.},
  year = {1978}
}

@inproceedings{GeuversHurkens2017,
  address = {{Berlin}},
  series = {Lecture {{Notes}} in {{Computer Science}}},
  title = {Deriving Natural Deduction Rules from Truth Tables},
  volume = {10119},
  isbn = {978-3-662-54069-5},
  language = {en},
  booktitle = {Logic and {{Its Applications}}},
  publisher = {{Springer}},
  doi = {10.1007/978-3-662-54069-5_10},
  author = {Geuvers, Herman and Hurkens, Tonny},
  editor = {Ghosh, Sujata and Prasad, Sanjiva},
  year = {2017},
  pages = {123-138}
}

@article{Zach2017a,
  title = {General Natural Deduction Rules and General Lambda Calculi},
  volume = {23},
  copyright = {All rights reserved},
  number = {3},
  journal = {Bulletin of Symbolic Logic},
  author = {Zach, Richard},
  year = {2017},
  pages = {371}
}

@article{BaazLeitsch2000,
  title = {Cut-Elimination and Redundancy-Elimination by Resolution},
  volume = {29},
  issn = {0747-7171},
  language = {en},
  number = {2},
  journal = {Journal of Symbolic Computation},
  doi = {10.1006/jsco.1999.0359},
  author = {Baaz, Matthias and Leitsch, Alexander},
  year = {2000},
  pages = {149-176}
}
\end{document}